\documentclass{amsart}

\usepackage{amsmath,amssymb}
\usepackage{amsthm}
\usepackage{latexsym}
\usepackage{graphicx}
\usepackage{color}
\usepackage{subfigure}
\usepackage{epstopdf}
\usepackage{eufrak}
\usepackage{url}

\newcommand{\conceft}{\text{ConceFT}}
\newcommand{\diag}{\text{diag}}

\newcommand{\NN}{\mathbb{N}}

\newcommand{\RR}{\mathbb{R}}
\newcommand{\CC}{\mathbb{C}}

\newcommand{\ud}{\textup{d}}
\newcommand{\argmax}{\operatornamewithlimits{argmax}}

\newcommand{\oeps}{\overline{\epsilon}}

\theoremstyle{definition}
\newtheorem{theorem}{Theorem}[section]
\newtheorem{defn}[theorem]{Definition}

\newtheorem{prop}[theorem]{Proposition}
\newtheorem{lemma}[theorem]{Lemma}

\theoremstyle{remark}
\newtheorem*{remark}{Remark}

\title[$\conceft$]{$\conceft$: Concentration of Frequency and Time via a multitapered synchrosqueezed transform}

\author[Ingrid Daubechies]{Ingrid Daubechies}
\address{Mathematics, Duke University} 
\author[Yi (Grace) Wang]{Yi (Grace) Wang}
\address{Mathematics, Syracuse University}
\author[Hau-Tieng Wu]{Hau-Tieng Wu}
\address{Mathematics, University of Toronto}
\email{hauwu@math.toronto.edu}

\begin{document}

\maketitle

\begin{abstract}
A new method is proposed to determine the time-frequency content of time-dependent
signals consisting
of multiple oscillatory components, with time-varying amplitudes and instantaneous frequencies.
Numerical experiments as well as a theoretical analysis are presented to assess its 
effectiveness.
\end{abstract}

\section{Introduction}\label{Section:Intro}

Oscillatory signals occur in a wide range of fields, including geophysics, biology, medicine, finance and social dynamics. They often consist of several different oscillatory components, the nature, time-varying behavior and interaction of which reflect properties of the underlying system. In general, we want to assess the number, strength and rate of oscillation of the different components constituting the signal, to separate noise from signal, and to isolate individual components; efficient and robust extraction of this information from an observed signal will help us better describe and quantify the underlying dynamics that govern the system. For each of the quantities of interest listed, we thus want an estimator that is consistent, that has (ideally) small variance and that produces results robust to different types of noise.

If the observed signal $f$ can be written as a finite sum of so-called harmonic components, i.e. $f(t)=\sum_{\ell} a_{\ell}\cos(2\pi\xi_{\ell} t+\delta_\ell)$, where $a_{\ell}>0$ (respectively $\xi_{\ell}>0$) represents the strength or amplitude (respectively frequency) of the $\ell$-th component, then one can recover the $a_{\ell}$ and $\xi_{\ell}$ from time-samples of $f(t)$ via the Fourier transform $\hat{f}$ of $f$, defined by $\hat{f}(\xi):=\int f(t)e^{-i2\pi \xi t}\ud t$. (If the $\xi_{\ell}$ are all integer multiples of a common $1/t_0$, then the integral can be taken over an interval of length $t_0$; when this is not the case, one can resort to integrals over long time intervals and average. Typically only discrete samples $f(t_n),\, n \in \mathbb{Z}$, are known, rather than the continuous time course $f(t),\, t \in \mathbb{R}$, and the integrals are estimated by quadrature.) However, oscillatory signals of interest often have more complex behavior. We shall be interested in particular in signals that are still the combination of ``elementary'' oscillations, but in which both the amplitude and the frequency of the components are no longer constant; they can be written as 
\begin{equation}\label{Introduction:eq1}
f(t)=\sum_{k=1}^KA_k(t)\cos(2\pi\varphi_k(t)),
\end{equation}
where $K\in\NN$, $A_k(t)>0$ and $\varphi'_k(t)>0$ for all $k$, but $A_k(t)$ and $\varphi'_k(t)$ are not constants. One can compute the Fourier transform $\hat{f}$ of such signals, and recover $f$ from $\hat{f}$ (this can be validly done for a much wider class of functions), but it is now less straightforward to determine the time-varying amplitudes $A_k(t)$ and the so-called ``instantaneous frequencies'' $\varphi'_k(t)$ from $\hat{f}$. Although the time-local behavior of the oscillations, and their deviation from perfect periodicity, cannot be captured by the Fourier transform in an easily ``readable'' way, an accurate description of this instantaneous behavior is nevertheless important in many applications, both to understand the system producing the signal and to predict its future behavior. Examples in the medical field include studies of the circadian \cite{Golombek_Rosenstein:2010,Takeda_Maemura:2011} and cortical rhythms \cite{Wang:2010}, or of heart-rate \cite{Ahmad_Tejuja_Newman_Zarychanski_Seely:2009,Lewis_Furman_McCool_Porges:2012,Lin_Wu_Tsao_Yien_Hseu:2014} and respiratory variability \cite{Wu_Hseu_Bien_Kou_Daubechies:2013,Seely_Bravi_Herry_Longtin_Ramsay_Fergusson_McIntyre:2014,Baudin_Wu_Bordessoule_Beck_Jouvet_Frasch_Emeriaud:2014}, all widely studied to understand physiology and predict clinical outcomes. 

The last 50 years have seen many approaches, in applied harmonic analysis and signal processing, to develop useful analysis tools for signals of this type; this is the domain of {\em time-frequency (TF) analysis}. Several algorithms and associated theories have been developed 
and widely applied (see, e.g., the overview \cite{Flandrin:1999}); well known examples include the short time Fourier transform (STFT), continuous wavelet transform (CWT) and Wigner-Ville distribution (WVD). The main idea is often to ``localize'' a portion of the signal in time, and then ``measure'' the oscillatory behavior of this portion. For example, given a function $f\in L^2$, the {\it windowed} or {\it short time Fourier transform} (STFT) associated with a window function $h(t)$ can be defined as:
\[
V_f^{(h)}(t,\eta):=\int f(s)h(t-s)e^{-i2\pi \eta (t-s)}\ud s
\]
where $t\in\RR$ is the time, $\eta\in \RR^+$ is the frequency, $h$ is the window function chosen by the user -- a commonly used choice is the Gaussian function with kernel bandwidth $\sigma>0$, i.e. $h(t)=(2\pi\sigma)^{-1/2}e^{-t^2/(2\sigma^2)}$. (The overall phase factor $e^{-i2\pi \eta t}$ is not always present in the STFT, leading to the name {\it modified short time Fourier transform} (mSTFT) for this particular form in \cite{Thakur_Wu:2011}.) 

Other, more specialized methods, targeting in particular signals of type (\ref{Introduction:eq1}), include the empirical mode decomposition \cite{Huang_Shen_Long_Wu_Shih_Zheng_Yen_Tung_Liu:1998}, ensemble empirical mode decomposition \cite{Wu_Huang:2009}, the sparsity approach \cite{Tavallali_Hou_Shi:2014}, iterative convolution-\ filtering \cite{Lin_Wang_Zhou:2009,Huang_Wang_Yang:2009}, the approximation approach \cite{Chui_Mhaskar:2015}, non-local mean approach \cite{Galiano_Velasco:2014}, time-varying autoregression and moving average approach \cite{DeLivera_Alysha_Hyndman_Snyder:2011} as well as the synchrosqueezing transforms introduced and studied by some of us \cite{Daubechies_Maes:1996, Daubechies_Lu_Wu:2011, Wu_Flandrin_Daubechies:2011, Wu:2011Thesis, Thakur_Wu:2011}.  

All TF methods that target reasonably large classes of functions (as opposed to functions with such specific models that complete characterization requires only fitting a small number of parameters) must face the Heisenberg uncertainty principle, limiting how accurately oscillatory information can be captured over short time intervals; for toy signals specially designed to have precise TF properties (e.g., chirps), this typically expresses itself by a ``blurring'' or ``smearing out'' of their TF representation, regardless of the analysis tool used. {\it Reassignment methods} \cite{Kodera_Gendrin_Villedary:1978,Auger_Flandrin:1995,Chassande-Mottin_Auger_Flandrin:2003}, introduced in 1978 and recently attracting more attention again, were proposed to analyze and possibly counter this. Their main idea is to analyze the local behavior in the TF plane of portions of the  representation, and determine nearby possible TF concentration candidates that best explain it; each small portion  is then reallocated to its ``right'' place in the TF plane, to obtain a more concentrated TF representation that, one hopes, gives a faithful and precise rendering of the TF properties of the signal. Reassignment methods can be applied to very general TF representations \cite{Auger_Flandrin:1995,Flandrin:1999}; they can be adaptive as well \cite{Auger_Chassande-Mottin_Flandrin:2012}. It has been argued recently \cite{Galiano_Velasco:2014} that reassignment methods can be viewed as analogs of ``non-local means'' techniques commonly applied in image processing; this provides an intuitive explanation for their robustness to noise.

The synchrosqueezing transform (SST) can be viewed as a special reassignment method \cite{Auger_Flandrin:1995, Chassande-Mottin_Auger_Flandrin:2003,Auger_Chassande-Mottin_Flandrin:2012}. In SST, the STFT or CWT coefficients are reassigned {\it only} in the frequency ``direction'' \cite{Daubechies_Lu_Wu:2011,Wu:2011Thesis,Thakur:2014}; this preserves causality, making it possible to reconstruct each component with real-time implementation \cite{Chui_Lin_Wu:2014}. The STFT-based SST of $f$ is defined as
\[
S_f^{(h)}(t,\xi):=\lim_{\alpha \rightarrow 0} \int V_f^{(h)}(t,\eta)\, g_\alpha(\xi-\omega^{(h)}_f(t,\eta))\, \ud \eta,
\]
where $g_\alpha$  is an ``approximate $\delta$-function'' (i.e. $g$ is smooth and has fast decay, with $\int g(x)\ud x =1$, so that $g_\alpha(t):=\frac{1}{\alpha}g(\frac{t}{\alpha})$ tends weakly to the delta measure $\delta$ as $\alpha\to 0$), and with $\omega^{(h)}_f$ defined by
\[
\omega^{(h)}_f(t,\eta):= \frac{-i\partial_t V_f^{(h)}(t,\eta)}{2\pi V_f^{(h)}(t,\eta)} \,\mbox{ if }\, V_f^{(h)}(t,\eta)\neq 0\,,\, \mbox{ and } 
\,\omega^{(h)}_f(t,\eta):= -\infty \, \mbox{ otherwise}\nonumber.
\]
The SST for CWT is defined similarly; see \cite{Daubechies_Lu_Wu:2011,Chen_Cheng_Wu:2014}, or Section \ref{Section:Conceftalgorithm}. SST was proposed originally for sound signals \cite{Maes:1995,Daubechies_Maes:1996}; its theoretical properties have been studied extensively \cite{Daubechies_Lu_Wu:2011,Wu:2013, Chen_Cheng_Wu:2014,Meignen_Oberlin_McLaughlin:2012, Chui_Lin_Wu:2014,Thakur:2014,Lin_Flandrin_Wu:2015}, including its stability to different types of noise \cite{Brevdo_Fuckar_Thakur_Wu:2012,Chen_Cheng_Wu:2014}. Several variations of the SST have been proposed \cite{Li_Liang:2012,Meignen_Oberlin_McLaughlin:2012, Yang:2014,Oberlin_Meignen_Perrier:2015,Xi_Cao_Chen_Zhang_Jin:2015}; in particular, the SST-approach can also be used for other TF representations, such as the wave packet transform \cite{Yang:2014}, and it can be extended to two-dimensional signals (such as images) \cite{Yang_Ying:2013,Yang_Ying:2014}. In addition, its practical usefulness has been demonstrated in a wide range of fields, including medicine \cite{Stankovski_Duggento_McClintock_Stefanovska:2012,Iatsenko_Bernjak_Stankovski_Shiogai_Owen_Clarkson_McClintock_Stefanovska:2013,Lin_Wu_Tsao_Yien_Hseu:2014, Wu_Hseu_Bien_Kou_Daubechies:2013, Wu_Chan_Lin_Yeh:2014, Baudin_Wu_Bordessoule_Beck_Jouvet_Frasch_Emeriaud:2014,Wu_Talmon_Lo:2015,Lin:2015Thesis}, mechanics \cite{Li_Liang:2012a,Feng_Chen_Liang:2015,Xi_Cao_Chen_Zhang_Jin:2015}, finance \cite{Guharay_Thakur_Goodman_Rosen_Houser:2013,Vatter_Wu_Chavez-Demoulin_Yu:2013}, geography \cite{Wang_Gao_Wang:2014, Herrera_Han_vanderBaan:2014, Tary_Herrera_Han_vanderBaan:2014}, denoising \cite{Meignen_Oberlin_McLaughlin:2012}, atomic physics \cite{Li_Sheu_Laughlin_Chu:2013,Sheu_Hsu_Wu_Li_Chu:2014,Li_Sheu_Laughlin_Chu:2015} and image analysis \cite{Yang_Lu_Ying:2014,Yang_Lu_Brown_Daubechies_Ying:2014}.

The SST approach can extract the instantaneous frequency and reconstruct the constitutional oscillatory components of a signal of type (\ref{Introduction:eq1}) in the presence of noise \cite{Brevdo_Fuckar_Thakur_Wu:2012,Chen_Cheng_Wu:2014}. However, its performance suffers when SNR gets low: as the noise level increases, and even before it completely obscures the main concentration in the TF plane of the signal, spurious concentration areas appear elsewhere in the TF plane, caused by correlations introduced by the overcomplete STFT or CWT analysis tool. The effect of these misleading perturbations, which downgrade the quality of the results, can be countered, to some extent, by {\it multi-tapering}.

Multi-tapering is a technique originally proposed to reduce the variance and hence stabilize power spectrum estimation in the spectral analysis of stationary signals \cite{Thomson:1982, Percival:1993, Babadi_Brown:2014}. Sampling the signal during only a finite interval leads to artifacts, traditionally reduced by tapering; an unfortunate side effect of tapering is to diminish the impact of samples at the extremes of the time interval. Thomson \cite{Thomson:1982} showed that one can nevertheless exploit optimally the information provided by the samples at the extremities, by using several orthonormal functions as tapers: the average of the corresponding power spectra is a good estimator with reduced variance. This technique has since been applied widely \cite{Percival:1993,Fraser_Boashash:1994,Farry_Buaniuk_Walker:1995,Bayram_Baraniuk:1996,Xu_Haykin_Racine:1999}. Multi-tapering was later extended to non-stationary TF analysis by combining it with reassignment \cite{Xiao_Flandrin:2007,Lin_Hseu_Yien_Tsao:2011,Orini_Bailon_Mainardi_Laguna_Flandrin:2012}: a more robust ``combined'' reassigned TF representation is obtained by picking orthonormal ``windows'' (used to isolate portions of the TF representation when working with a reassignment method), and averaging the reassigned TF representations determined by each of the individual windows. Heuristically, the concentration for a ``true'' constituting component of the signal will be in similar locations in the TF plane for each of the individual reassigned TF representations, whereas the spurious concentrations, artifacts of correlations between noise and the windowing function, tend to not be co-located and have a diminished impact when averaged. In the SST context a similar multi-taper idea was used by one of us in a study of anesthesia depth \cite{Lin_Wu_Tsao_Yien_Hseu:2014,Lin:2015Thesis}, in which $J$ different window functions $h_j,\, j=1,\ldots,J$ are considered, and the multi-taper SST (MTSST) is computed as follows:
\[
\texttt{MS}_f(t,\xi):=\frac{1}{J}\sum_{j=1}^J S_f^{(h_j)}(t,\xi).
\]
Using multiple tapers reduces artifacts, and the MTSST remains ``readable'' at higher noise levels than a ``simple'' SST \cite{Lin_Wu_Tsao_Yien_Hseu:2014,Lin:2015Thesis}. To optimally suppress noise artifacts it is tempting to consider increasingly larger $J$. However, the area in the TF plane over which the signal TF information is ``smeared out'' also increases (linearly) with $J$, and a balance needs to be observed; in the multi-taper reassignment method of \cite{Xiao_Flandrin:2007}, for instance, 6 Hermite functions were used (i.e. $J=6$.)

In this paper, we introduce a new approach to obtain better concentrated time-frequency representations, which we call {\it $\conceft$}, for {\it {\bf Conce}ntration in {\bf F}requency and {\bf T}ime}. It is based on STFT- or CWT-based SST, but the approach could be applied to yet other TF decomposition tools. The $\conceft$ algorithm will be defined precisely below, in Section \ref{Section:Conceftalgorithm}. Like MTSST, $\conceft$ starts from a multi-layered time-frequency representation, but instead of averaging the SST results obtained from STFT or CWT for orthonormal windows, which can be viewed as elements in a vector space of time-frequency functions, it considers many different projections in this same vector space, and averages the corresponding SSTs; for more details, see Section \ref{Section:Conceftalgorithm}. Section \ref{Section:Concefttheory} studies the theoretical properties of $\conceft$, and explains how it can provide reliable results under challenging SNR conditions; finally, in Section \ref{Section:Conceftnumerical}, we provide several numerical results.

To conclude this introduction, we illustrate $\conceft$ on a simulated signal, in which the clean signal $s(t)$ is composed of two oscillatory components: $s(t)=s_1(t)+s_2(t)$, where $s_1(t)=A_1(t)\cos(2\pi\varphi_1(t))\chi_{[3,12]}(t)$, and  $s_2(t)=A_2(t)\cos(2\pi\varphi_2(t))\chi_{[0,8]}(t)$ (here $\chi$ stands for the indicator function, $\chi_{[a,b]}(t) = 1$ if $a \leq t \leq b$, $\chi_{[a,b]}(t) = 0$ otherwise); $A_i(t)>0$ and $\varphi_i'(t)>0$ for $i=1,2$. This signal is sampled at rate $100$Hz, from $t=0$ to $t=12$ seconds. To these signal samples we add independent realizations of a fat-tailed noise $\xi$, which is identically-t4-Student-distributed with variance $2.036$. The left panels in Figure \ref{fig:Introduction:Example1} show the three constituents of the total (noisy) signal $Y(t)=s_1(t)+s_2(t)+\xi(t)$; note that each of $s_1$ and $s_2$ ``lives'' during only part of the full time observation interval; the fat-tailed nature of the noise causes the bursty behavior evident in the plot of $\xi(t)$. The individual plots of the $s_i$ show the amplitude modulations $A_i(t)$ of the $s_i$; Figure \ref{fig:Introduction:Example1} also graphs $\varphi_i'(t)>0$ for $i=1,2$. In addition, Figure \ref{fig:Introduction:Example1} shows the time course of both the clean signal $s(t)$ and the noisy signal $Y(t)$, at the same scale; their signal-to-noise ratio is $-0.85$, computed as $20\log_{10}\left(\frac{\text{std}(s(t))}{\text{std}(\xi(t))}\right)$, where $\text{std}$ stands for standard deviation. Figure \ref{fig:Introduction:Example2} shows several SST-based results for this (quite challenging) example. For the clean signal $s$, the ``mono-SST'' (STFT-based, with a Gaussian window) performs quite well, with only some artifacts at the onset and cessation of the $s_i$; many structured artifacts are visible in the mono-SST of the noisy signal $Y$. Both MTSST and $\conceft$ remove the onset and cessation artifacts for the clean $s$ (shown only for $\conceft$ in the figure, but similar for MTSST); the improvement is much more marked for the noisy signal $Y$: the spurious ``bubbles'' are suppressed to some extent in the MTSST-based representation (using 2 orthonormal windows: the same Gaussian and the next higher-order Hermite function); a more dramatic improvement is seen in the $\conceft$-representation corresponding to the same vector space of windows.

\begin{figure}[h!]
\begin{centering}
\includegraphics[width= \textwidth]{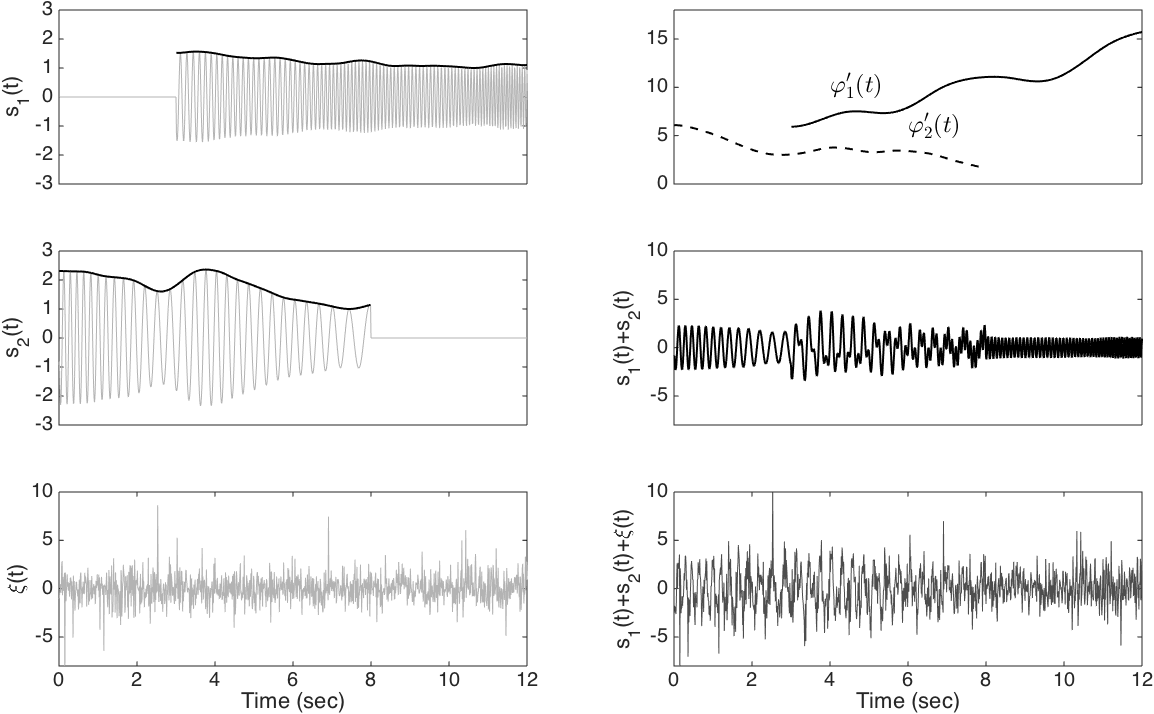}
\end{centering}

\caption{\label{fig:Introduction:Example1}Left panels: the three constituents of the noisy signal $Y(t)$: oscillatory components $s_1(t)$ (top), 
and $s_2(t)$ (middle), and the bursty iid t4-Student noise $\xi(t)$ (bottom). Note
that $s_1(t)\neq 0$ only for $t>3$, $s_2(t)\neq 0$ only for $t<8$ sec.; their
respective amplitudes $A_i(t)$ are plotted as envelopes for each. Right panels: plots of $\varphi'_1(t)$ (solid) and $\varphi'_2(t)$ (dashed) in top panel; the clean
signal $s=s_1+s_2$ (middle) and noisy signal, $Y(t)=s(t)+\xi(t)$ (bottom), plotted with the same scale.}
\end{figure}

\begin{figure}[h!]
\begin{centering}
\includegraphics[width= \textwidth]{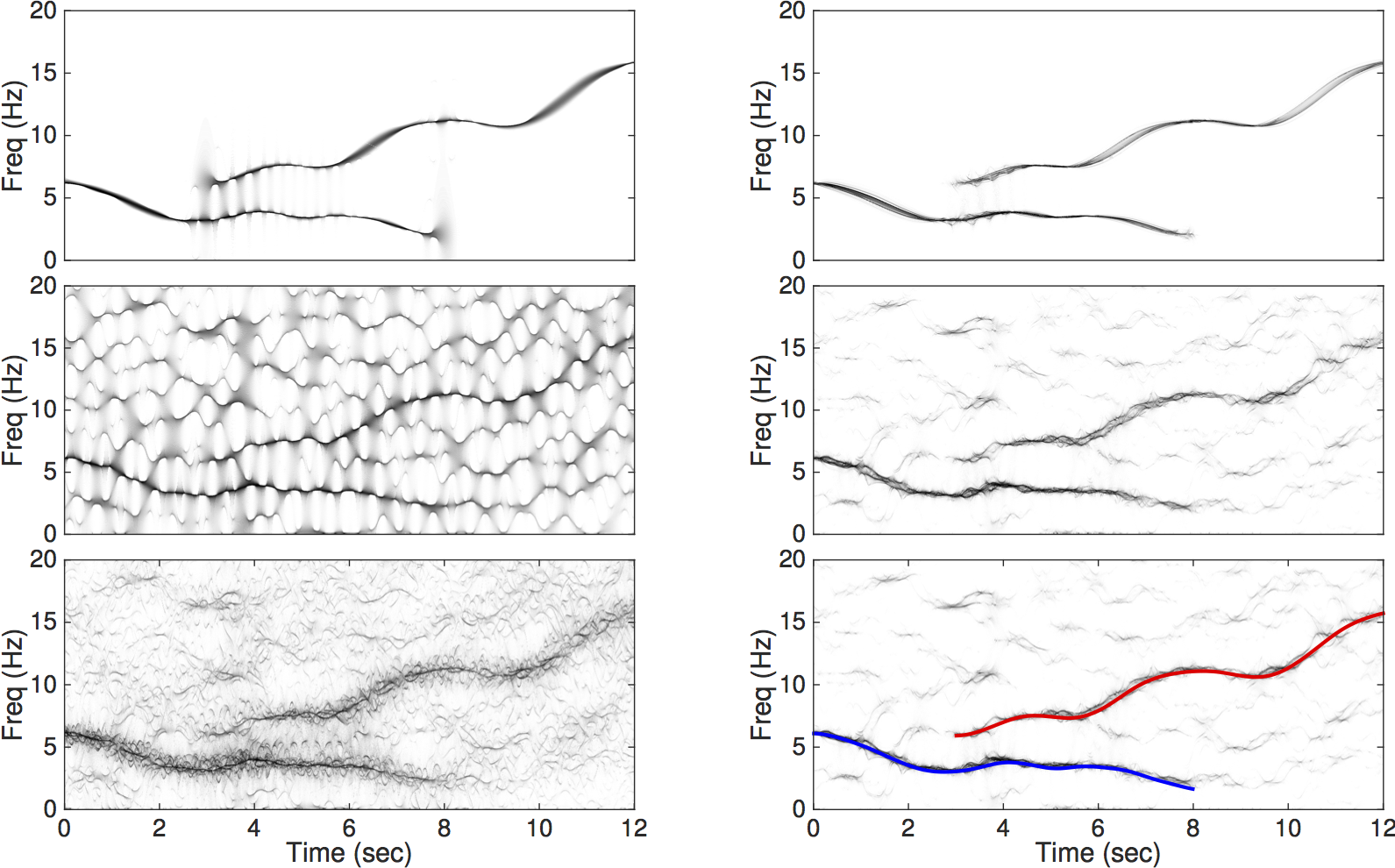}
\end{centering}

\caption{\label{fig:Introduction:Example2}Top left: STFT-based synchrosqueezing transform (SST) of the clean signal $s(t)$ for a Gaussian window $h$; middle
left: STFT-based SST of the noisy signal $Y(t)$, with the same window. Bottom
left: multi-taper SST of $Y(t)$, choosing the Gaussian and the next 5 Hermite functions as windows; this result is closer to the STFT-based SST of $s$. Top right: $\conceft$ of $s(t)$ based on the
same two Hermite functions; middle right: $\conceft$ of $Y(t)$ based on the
two Hermite functions; bottom right: same as middle right, with plots of
$\varphi'_1(t)$ and $\varphi'_2(t)$ superimposed.
}
\end{figure}

\section{The $\conceft$ algorithm}\label{Section:Conceftalgorithm}

We start by briefly reviewing SST. In the introduction, we defined STFT-based SST, discussed in more detail in \cite{Wu:2011Thesis,Thakur:2014}; to show that the situation is very similar for CWT-based SST, we discuss that case here; see \cite{Daubechies_Lu_Wu:2011,Chen_Cheng_Wu:2014} for details. We start with the wavelet $\psi$ with respect to which the CWT will be computed, which must necessarily have mean zero; that is, $\int \psi(t)\,dt =0$; let's also pick it to be a Schwartz function. 
We shall assume that we are dealing with real signals $f$; in this case the symmetry in $\xi$
of $\widehat{f}(\xi)$ makes it possible to consider only the ``positive frequency part'' of $f$, by
picking $\psi$ so that its Fourier transform $\widehat{\psi}$ is supported on $\mathbb{R}_{+}$. (The approach can be extended easily to handle complex signals as well, but notation becomes a bit heavier.)
Then the {\it Continuous Wavelet Transform} $W^{(\psi)}_f(a,b)$ of a tempered distribution $f$, with the variables $a$, $b$ standing for scale and time location, is defined as the inner product of $f$ with $\psi^{(a,b)}(t)=|a|^{-1/2}\psi((t-b)/a)$. 
Even if the Fourier transform $\hat{f}$ is very concentrated around some frequency $\omega_0$, the magnitude $|W^{(\psi)}_f(a,b)|$ of the CWT will be spread out over a range of scales $a$, corresponding to a neighborhood of $\omega_0$. However, the phase information of $W^{(\psi)}_f$ will still hold a ``fingerprint'' of $\omega_0$ on that whole neighborhood, in that $W^{(\psi)}_f(a,b)$ will show oscillatory behavior in $b$, with frequency $\omega_0$, for a range of different $a$. This is the motivation for the synchrosqueezing transform,  which shifts the CWT coefficients ``back'', according to certain reassignment rules determined by the phase information. More concretely, we set a {\it threshold} $\Gamma>0$, and then define 
\begin{align*}
\Omega^{(\psi,\Gamma)}_f(a,b):= \left\{
\begin{array}{ll}
\displaystyle\frac{-i\partial_b W^{(\psi)}_{f}(a,b)}{2\pi W^{(\psi)}_{f}(a,b)} & \mbox{when }|W^{(\psi)}_{f}(a,b)|>\Gamma\\
-\infty & \mbox{when }|W^{(\psi)}_{f}(a,b)|\leq \Gamma.
\end{array}
\right.
\end{align*}
where $\partial_b$ is the partial derivative with respect to $b$ (see the {\em Electronic
Supplementary Materials} -- or ESM -- for a remark
concerning robust numerical implementation); the hard threshold $\Gamma$ can be adjusted for best reduction of the numerical error and noise influence. The CWT-based synchrosqueezed transform (or CWT-based SST)  then moves the CWT coefficient $W^{(\psi)}_f(a,b)$ to the ``right'' frequency slot, using 
$\Omega^{(\psi,\Gamma)}_f(a,b)$ as guideline: 
\[
S^{(\psi,\Gamma,\alpha)}_{f}(b,\xi):=
\int_{\{a:\,|W^{(\psi)}_f(a,b)|>\Gamma\}}
W^{(\psi)}_f(a,b)\frac{1}{\alpha}\,g\!
\left(\frac{\xi-\Omega^{(\psi,\Gamma)}_f(a,b)}
{\alpha}\right)a^{-3/2}\ud a,
\]
where $0<\alpha\ll 1$ is chosen by the user, $g$ is a smooth function so that $\frac{1}{\alpha}g(\frac{\cdot}{\alpha})\to \delta$ in the weak sense as $\alpha\to 0$, and the factor $a^{-3/2}$ is introduced to ensure that the integral of $S^{(\psi,\Gamma,\alpha)}_{f}(b,\xi)$ over $\xi$ yields a close approximation to the original $f(b)$. For more details, we refer the reader to \cite{Daubechies_Lu_Wu:2011,Chen_Cheng_Wu:2014}.

Although both the CWT $W^{(\psi)}_f$ and its derived SST $S^{(\psi,\Gamma,\alpha)}_{f}$ depend on the choice of the reference wavelet $\psi$, this is much less pronounced for the SST; CWT-based SST corresponding to different reference wavelets lead to different but very similar TF representations. (Theoretical reasons for this can be found in \cite{Daubechies_Lu_Wu:2011,Chen_Cheng_Wu:2014}.) In particular, the dominant components in the TF representations are very similar. Moreover, even when the signal is contaminated by noise, these dominant components in the TF representations are not significantly disturbed \cite{Chen_Cheng_Wu:2014}. However, the distribution of artifacts across the TF representation, induced by the noise, as seen in e.g. the middle left panel of Figure \ref{fig:Introduction:Example2}, vary from one reference wavelet to another; this can be intuitively explained by observing that the CWT is essentially a convolution with (scaled versions of) the reference wavelet, so that the wavelet transforms of i.i.d. noise based on different {\em orthogonal} reference wavelets are {\em independent}. These observations lead to the idea of a multi-taper SST algorithm \cite{Lin_Wu_Tsao_Yien_Hseu:2014,Lin:2015Thesis}. In brief, given $J$ orthonormal reference wavelets $\psi_j$, $j=1,\ldots,J$, one determines the reassignment rules $\Omega^{(\psi_j, \Gamma)}_f(a,b)$, as well as the corresponding $S^{(\psi,\Gamma,\alpha)}_{f}(b,\xi)$, and then defines the MTSST by
\[
\texttt{MS}^{\Gamma,\alpha}_f(b,\xi):=\frac{1}{J}\sum_{j=1}^J S^{(\psi_j,\Gamma,\alpha)}_{f}(b,\xi).
\]
This suggests that averaging over a large number of orthonormal reference wavelets would smooth out completely the TF artifacts induced by the noise, as originally discussed for the reassignment method \cite{Xiao_Flandrin:2007}. However, in order for reassignment to make sense, the reference function, whether it is the window $h$ for STFT or the wavelet $\psi$ for CWT, must itself be fairly well concentrated in time and frequency, so that inner products with modulated window functions or scaled wavelets do not mix up different components and behaviors of the signal. On the other hand, there is a limit to how many orthonormal functions can be ``mostly'' supported in a concentrated region in the TF-plane -- by a rule of thumb generalizing the Nyquist sampling density one can find, for a region $\mathcal{R}$ in the TF-plane, only $\mbox{Area}(\mathcal{R})/(2\pi)$ orthonormal functions that are mostly concentrated on $\mathcal{R}$ \cite{Daubechies:1992}. This limits how many different orthonormal $\psi_j$ can be used in MTSST. 

$\conceft$ uses the different TF ``views'' provided by the CWT transforms $W^{(\psi_j)}_f$ in a different way, exploiting the {\it non-linearity} of the SST operation. (See the ESM for a sketch of an alternate way in which one could extend multi-taper CWT, not pursued in this paper, however.) For each choice of $\psi$, the collection of CWT $W^{(\psi)}_f$, where $f$ ranges over the class of signals of interest, span a subspace of $\mathcal{F}$, the space of all reasonably smooth functions of the two variables $a$, $b$. Different orthonormal $\psi_j$ generate different subspaces in $\mathcal{F}$; combined, they generate a larger subspace, in which one can define an infinite number of ``sections'', each corresponding to the collection of CWT generated by one reference wavelet. Each linear combination of the $\psi_j$ defines such a CWT-space, in which one can carry out the corresponding SST operation. For $\psi = \sum_{j=1}^J r_j \, \psi_j$, where $r_j\in\RR$, one has $W^{(\psi)}_f=\sum_{j=1}^J r_j\, W^{(\psi_j)}_f$; because synchrosqueezing is a highly nonlinear operation, the corresponding $S^{(\psi,\Gamma,\alpha)}_{f}$ are however not linear combinations of the $S^{(\psi_j,\Gamma,\alpha)}_{f}$. In practice, the artificial concentrations in the TF-plane, triggered by fortuitous correlations between the noise and the (overcomplete) $\psi^{(a,b)}$, occur at locations sufficiently different, for different choices of the vector $\boldsymbol{r}=(r_1,\ldots,r_J)$, that averaging over many choices of $\boldsymbol{r}$ successfully suppresses noise artifacts. 

More precisely, the CWT-based $\conceft$ algorithm proceeds as follows:
\begin{itemize}
\item Take $J$ orthonormal reference wavelets, $\psi_1,\ldots,\psi_J$, in the Schwartz space, with good concentration in the TF-plane. 
\item Pick $N$ random vectors $\boldsymbol{r}_n$, $n=1, \ldots, N$, of unit norm, in $\mathbb{R}^J$; that is, uniformly select $N$ samples in $S^{J-1}$.
\item For each $n$ between 1 and $N$, define $\psi_{[n]}:=\sum_{j=1}^J (\boldsymbol{r}_n)_j \,\psi_j$, and $W^{(\psi_{[n]})}_f=\sum_{j=1}^J (\boldsymbol{r}_n)_j\, W^{(\psi_j)}_f$.
\item Select the threshold $\Gamma >0$ and the approximation parameter $\alpha>0$, and evaluate, for each $n$ between 1 and $N$, the corresponding CWT-based SST of $f$ by computing the reassignment rule $\Omega^{(\psi_{[n]},\Gamma)}_f(a,b)$, and hence $S^{(\psi_{[n]}, \Gamma,\alpha)}_f(b,\xi)$, as defined above, with the minor adjustment that when the expression $\sum_{j=1}^J(\boldsymbol{r}_n)_j\, W^{(\psi_j)}_f(a,b)$ in the reassignment rule denominator has a negative real part, we switch to the vector $-\boldsymbol{r}_n$. 
\item The final $\conceft$ representation of $f$ is then the average
\begin{equation}
C^{\Gamma,\alpha}_f(b,\xi):= \frac{1}{N}\sum_{n=1}^N S^{(\psi_{[n]},\Gamma, \alpha)}_f(b,\xi).
\end{equation}
In practice, $J$ could be as small as $2$, while $N$ could be chosen as large as the user wishes. 
\end{itemize}
The square of the magnitude of $C^{\Gamma,\alpha}_f(b,\xi)$, 
\[
\widetilde{\texttt{P}}_f(b,\xi):=|C^{\Gamma,\alpha}_{f}(b,\xi)|^2\,,
\]
can be of interest in its own right, as an estimated {\it time-varying Power Spectrum} (tvPS) of $f$. 

STFT-based $\conceft$ representations are defined entirely analogously, based on the STFT-reassignment rule given in Section \ref{Section:Intro}. 

\section{Theoretical Results}\label{Section:Concefttheory}

In this section, we list and explain theoretical results about CWT-based $\conceft$. The detailed mathematical computations and proofs can be found in the ESM. Entirely similar results hold for STFT-based $\conceft$; since they are established by the same arguments, we skip those details. We start by recalling the structure of our signal space, as introduced in \cite{Daubechies_Lu_Wu:2011, Chen_Cheng_Wu:2014}. We emphasize that this is, to a large extent, a purely phenomenological model, constructed so as to reflect the fairly (but not exactly) periodic nature of many signals of interest, in particular (but not only) those of a physiological origin (see the discussion in \cite{Wu_Hseu_Bien_Kou_Daubechies:2013}). 

A single-component or {\it intrinsic-mode type} (IMT) function has the 
following form:
\[
F(t) = A(t)\cos(2\pi\varphi(t)),
\]
where the amplitude modulation $A(t)$ and the phase function $\varphi(t)$ are both reasonably smooth; in addition, both $A(t)$ and the derivative $\varphi'(t)$ (or the ``instantaneous frequency'') are strictly positive at all time as well as bounded; finally, we assume that $A$ and $\varphi'$ vary in time at rates that are slow compared to the instantaneous frequency of $F$ itself. For the precise mathematical formulation of these conditions we refer to the ESM; this precise formulation invokes a few parameters, one of which, $\epsilon$, bounds the ratio of the rate of change of $A$ and $\varphi'$. This parameter will play a role in our estimates below. [Although we are assuming the signal to be real-valued here, all this can easily be adapted to the complex case by replacing the cosine with the corresponding complex exponential; the discussion in the remainder of this section can be adapted similarly.]
We also consider signals that contain several IMT components, that is,
functions of the type
\begin{equation}
G(t)=\sum_{\ell=1}^L F_\ell(t)
=\sum_{\ell=1}^L A_\ell(t)\cos(2\pi\varphi_\ell(t)),
\end{equation}
where each $F_\ell$ is an IMT function, and we assume in addition that
the instantaneous frequencies $\varphi_\ell'(t)$ are ordered (higher $\ell$
corresponding to larger $\varphi_\ell'$) and well-separated,  
\begin{align}\label{definition:adaptiveHarmonicMultiple}
\varphi_{\ell+1}'(t)-\varphi'_\ell(t)>d(\varphi_{\ell+1}'(t)+\varphi_{\ell}'(t))
\end{align}
for all $\ell=1,\ldots,L-1$, for some $d$ with $0<d<1$. 
We denote by $\mathcal{A}$ the set of all such functions $G$; it provides
a flexible  {\it adaptive harmonic model} space for a wide class of signals
of interest. (Strictly speaking, they are not ``truly'' harmonic,
if harmonicity is interpreted -- as it often is -- 
as ``having components with frequencies that are 
integer multiples of a fundamental frequency''.)

Next, we turn to the noise model for which we prove our main theoretical result. For the purposes of this theoretical discussion, we use a simple additive Gaussian white noise (even though, as illustrated by the figures in the introduction, the approach works for much more challenging noise models as well!). That is, we consider our noisy signals to be of the form 
\begin{equation}\label{decompAdaptive}
Y(t) = G(t)+ \sigma\Phi(t) = \sum_{\ell=1}^L F_\ell(t) + \sigma\Phi(t) = \sum_{\ell=1}^L A_\ell(t)\cos(2\pi\varphi_\ell(t))+\sigma\Phi(t),
\end{equation}
where $G =\sum_{\ell=1}^L F_\ell$ is in $\mathcal{A}$, $\Phi$ is a Gaussian white noise with standard deviation $1$ and $\sigma>0$ is the noise level. Note that typically $Y$ is a generalized random process, since by definition $G$ is a tempered distribution. We could extend this, introducing also the trend and a more general noise model as in \cite{Chen_Cheng_Wu:2014}, the wave-shape function used in \cite{Wu:2013}, or the generalized IMT functions that model oscillatory signals with fast varying instantaneous frequency  of \cite{Kowalski_Meynard_Wu:2015}. None of these generalizations would significantly affect the mathematical analysis, but to simplify the discussion, we restrict ourselves to the model (\ref{decompAdaptive}).

Finally, we describe the wavelets $\psi_1, \ldots,\psi_J$ with respect to which we compute the CWT of $Y$. For the sake of convenience of the theoretical analysis, we assume that they are smooth functions with fast decay, and that their Fourier transforms $\widehat{\psi_j}$ are all real functions with compact support, $\mbox{supp}\widehat{\psi_j}\subset [1-\Delta_j,1+\Delta_j]$, where $0<\Delta_j<1$. We also assume that the $\psi_1,\ldots,\psi_J$ form an orthonormal set, that is, $\int\psi_i(x)\overline{\psi_j(x)}\ud x= \delta_{i,j}$, where $\delta_{i,j}$ is the Kronecker delta. To build appropriate linear combinations of the $\psi_j$, we define, for any unit-norm vector $\boldsymbol{r}=(r_1,\ldots,r_J)$ in $\mathbb{R}^J$, the corresponding combination as $\psi^{[\boldsymbol{r}]}:=\sum_{j=1}^J r_j \psi_j$.
It is convenient to characterize intervals for the scale $a$ such that
the support of $\widehat{\psi_j^{(a,b)}}$ overlaps $\varphi_\ell'(b)$, where $\psi_{j}^{(a,b)}(t):=\frac{1}{\sqrt{a}}\psi_j\left(\frac{t-b}{a}\right)$; we
thus introduce the notation  
$Z^{(j)}_\ell(b)=
\left[(1-\Delta_j)/\varphi_\ell'(b),(1+\Delta_j)/\varphi_\ell'(b)\right]$.
It then follows from the definition of the CWT as the inner product between
the signal and scaled, translated versions of the wavelets that (see   
\cite{Daubechies_Lu_Wu:2011, Chen_Cheng_Wu:2014})
\[
W^{(\psi_j)}_{F_\ell}(a,b)=\left\{
\begin{array}{ll}
e^{i2\pi\varphi_\ell(b)}Q_{j,\ell}(a,b)+
\epsilon_j(a,b)&\mbox{ when }a\in Z^{(j)}_\ell(b)\\
\epsilon_j(a,b)&\mbox{ otherwise},
\end{array}
\right.
\]
where
\begin{equation}\label{Expansion:Qk}
Q_{j,\ell}(a,b)=A_\ell(b)\sqrt{a}\,
\overline{\widehat{\psi_j}(a\varphi'_\ell(b))}\in \RR
\end{equation}
and $\epsilon_j$ is of order $\epsilon$ for all $j=1,\ldots,J$. Here $\epsilon_j$ depends on the first three absolute moments of $\psi_j$ and $\psi'_j$ and the model parameters. It follows that the wavelet transform of $Y$, with respect to $\psi_j$, is given by
\[
W^{(\psi_j)}_Y(a,b)=\sum_{\ell=1}^L
e^{i2\pi\varphi_\ell(b)}Q_{j,\ell}(a,b)\chi_{Z^{(j)}_\ell(b)}+
\epsilon_j(a,b)+\sigma\Phi(\psi_{j}^{(a,b)}),
\]
where $\chi_{Z^{(j)}_\ell(b)}$ is the indicator function of the set $Z^{(j)}_\ell(b)$; note that the $\epsilon_j(a,b)$-term, again of order $\epsilon$, need not be the same as before. As shorthand notations, we will use bold symbols to regroup quantities indexed by $j=1,\ldots,J$ into one $J$-dimensional vector, e.g. $\boldsymbol{\epsilon}(a,b)=[\epsilon_1(a,b),\ldots,\epsilon_J(a,b)]^\intercal$ (which has norm of order $\epsilon$), $\boldsymbol{\Phi}(a,b)=[\Phi(\psi_1^{(a,b)}),\ldots,\Phi(\psi_J^{(a,b)})]^\intercal$ (a complex Gaussian random vector \cite{Gallager:2008}, with mean $[0,\ldots,0]^\intercal\in\RR^J$, and covariance as well as relation matrix equal to $I_{J\times J}$ -- see ESM), $ \boldsymbol{W}^{\boldsymbol{\psi}}_Y(a,b)= [W^{(\psi_1)}_Y(a,b),\ldots,W^{(\psi_J)}_Y(a,b)]^\intercal$, and $\boldsymbol{Q}_\ell(a,b):=[Q_{1,\ell}(a,b),\ldots,Q_{J,\ell}(a,b)]^\intercal$. Finally, $W^{(\psi^{[\boldsymbol{r}]})}_Y(a,b):=\boldsymbol{r}^\intercal \boldsymbol{W}^{\boldsymbol{\psi}}_Y(a,b)$ or, more explicitly,
\begin{equation}
W^{(\psi^{[\boldsymbol{r}]})}_Y(a,b)=\sum_{\ell=1}^L \sum_{j=1}^J r_j
e^{i2\pi\varphi_j(b)} Q_{j,\ell}(a,b)\chi_{Z^{(j)}_\ell}(a,b)+
\boldsymbol{r}^\intercal\left[\boldsymbol{\epsilon}(a,b)+
\sigma\boldsymbol{\Phi}(a,b)\right].
\end{equation}
Under the general assumptions for our model, 
\[
-i\partial_b W^{(\psi^{[\boldsymbol{r}]})}_Y(a,b)
=\, 2\pi \left(\sum_{j=1}^J \sum_{\ell=1}^L 
r_j \varphi_j'(b) e^{i2\pi\varphi_j(b)} Q_{j,\ell}(a,b) \chi_{Z^{(j)}_\ell}(a,b)
+ \boldsymbol{r}^\intercal \left[\widetilde{\boldsymbol{\epsilon}}(a,b)
+ \sigma\widetilde{\boldsymbol{\Phi}}(a,b)\right]\right),
\]
where $\widetilde{\boldsymbol{\epsilon}}(a,b)$ is again a $J$-dimensional vector of order $\epsilon$, and $\widetilde{\boldsymbol{\Phi}}(a,b)$ is again a complex Gaussian random vector. The scalar products $\boldsymbol{r}^\intercal\boldsymbol{\Phi}(a,b)$ and $\boldsymbol{r}^\intercal\widetilde{\boldsymbol{\Phi}}(a,b)$ are independent complex Gaussian random variables, with mean 0 and variance $\|\boldsymbol{r}\|^2$, $ \sum_{j=1}^J r_j^2 \|\widehat{\psi'_j} \|^2/(2\pi a)^2$, respectively. (See ESM.) Set now $Z_\ell(b) = \cap_{j=1}^J Z^{(j)}_\ell(b)$. Then it follows that for $a \in Z_\ell(b)$, we get the following reassignment for the CWT $W^{(\psi^{[\boldsymbol{r}]})}_Y$:
$$
\omega_Y^{(\psi^{[\boldsymbol{r}]})}(a,b)
=\,\frac{-i\partial_b W^{(\psi^{[\boldsymbol{r}]})}_Y(a,b)}
{2 \pi W^{(\psi^{[\boldsymbol{r}]})}_Y(a,b)}
=\,\frac{\boldsymbol{r}^\intercal\left[\varphi_\ell'(b)e^{i2\pi\varphi_\ell(b)}\boldsymbol{Q}_\ell(a,b)
+\widetilde{\boldsymbol{\epsilon}}(a,b)+\sigma\widetilde{\boldsymbol{\Phi}}(a,b)\right]}
{\boldsymbol{r}^\intercal\left[e^{i 2\pi \varphi_\ell(b)}\boldsymbol{Q}_\ell(a,b)
+\boldsymbol{\epsilon}(a,b)+\sigma\boldsymbol{\Phi}(a,b)\right]},
$$
which is a ratio random variable of two dependent complex Gaussian random variables with non-zero means. 
Next, we consider, for each {\em fixed} realization of the random noise, the unit-norm vector $\boldsymbol{r}\in \mathbb{R}^J$ as a random vector, picked uniformly from $S_\kappa:=\left\{\boldsymbol{r}\in S^{J-1}\,;\, \left|\boldsymbol{r}^\intercal \boldsymbol{W}^{(\boldsymbol{\psi})}_Y(a,b)\right|>2\kappa\mbox{ and }\Re\left(\boldsymbol{r}^\intercal \boldsymbol{W}^{(\boldsymbol{\psi})}_Y(a,b)\right)>0\right\}\subset S^{J-1}$. Restricting the choice of $\boldsymbol{r}$ to the subset of $S^{J-1}$ for which the inner product of $\boldsymbol{r}$ and $\boldsymbol{W}^{\boldsymbol{\psi}}_Y(a,b)$ has magnitude larger than $2\kappa$ reflects the threshold used in the SST algorithm (see Section \ref{Section:Conceftalgorithm}); restricting $\boldsymbol{r}$ so that the inner product has positive real part means that we sample $\boldsymbol{r}$ from a half sphere rather than the whole sphere. (See ESM for more details.)
 
Assuming that the bound on the noise is such that $\|\boldsymbol{\epsilon}+\sigma\boldsymbol{\Phi}\|_2^2<\kappa$, then the expectation of $\omega_Y^{(\psi_{\boldsymbol{r}})}(a,b)$ over $S_\kappa$ is given by
\[
\mathbb{E}_{\boldsymbol{r}}\omega_Y^{(\psi_{\boldsymbol{r}})}(a,b)=
\varphi_\ell'(b)
+e^{-i2\pi\varphi_\ell(b)}\mathfrak{p}_{\boldsymbol{Q}_\ell(a,b)}
\left(\boldsymbol{V}_{\!\ell}(a,b)\right)
+E_1,\label{Lemma:ExpectationOverrOfOmega}
\]
where $\boldsymbol{V}_{\!\ell}(a,b):=
\widetilde{\boldsymbol{\epsilon}}(a,b)+\sigma\widetilde{\boldsymbol{\Phi}}(a,b)-\varphi_\ell'(b)[\boldsymbol{\epsilon}(a,b)+\sigma\boldsymbol{\Phi}(a,b)]$,
$\mathfrak{p}_{\boldsymbol{v}} $ denotes ``taking the component along'' a vector
$\boldsymbol{v} $, that is, $\mathfrak{p}_{\boldsymbol{v}}(\boldsymbol{u})
=\frac{\boldsymbol{v}^\intercal \boldsymbol{u}}{\|\boldsymbol{v}\|}$, and $E_1$ is bounded by
\[
|E_1|\leq
\frac{1}{2}\,
\left(\left[ 1-\frac{c}{J-1} \right] 
|\mathfrak{p}_{\boldsymbol{Q}_\ell(a,b)}\left(\boldsymbol{V}_{\!\ell}(a,b)\right)|^2 
+ c\, \frac{\|\boldsymbol{V}_{\!\ell}(a,b)\|^2}{J-1}\right)^{1/2}.
\]
Furthermore the variance is bounded by
\[
\text{Var}_{\boldsymbol{r}}\,\omega_Y^{(\psi^{[\boldsymbol{r}]})}(a,b)
\leq\,\frac{5}{2}\,\left(
\left[ 1-\frac{c}{J-1} \right] 
|\mathfrak{p}_{\boldsymbol{Q}_\ell(a,b)}
\left(\boldsymbol{V}_{\!\ell}(a,b) \right)|^2
+ c\, \frac{\|\boldsymbol{V}_{\!\ell}(a,b)\|^2}{J-1}\right)~. 
\]
A detailed derivation, and an explicit expression for the constant $c$ is given in the ESM; if $J$ becomes large, we have $c \approx 2 \sqrt{2}/[\kappa \sqrt{\pi J}]\,$. The quantity $\left|\mathfrak{p}_{\boldsymbol{Q}_\ell(a,b)}\left(\boldsymbol{V}_{\!\ell}(a,b) \right)\right|=\left|\mathfrak{p}_{\boldsymbol{Q}_\ell(a,b)}\left(\widetilde{\boldsymbol{\epsilon}}(a,b)+\sigma\widetilde{\boldsymbol{\Phi}}(a,b)-\varphi'_\ell(b)[\boldsymbol{\epsilon}(a,b)+\sigma\boldsymbol{\Phi}(a,b)]\right)\right|$, which occurs in several of these estimates, is, with high probability (with respect to the random noise process), fairly small for large $J$, because it is the norm of the projection onto $\boldsymbol{Q}_\ell(a,b)$ of $\boldsymbol{V}_{\!\ell}(a,b)$, and vectors that are unrelated (as is the case for $\boldsymbol{Q}_\ell(a,b)$ and $\boldsymbol{V}_{\!\ell}(a,b)$) have a higher chance of being close to orthogonal in higher dimensions. The other terms in $\mathbb{E}_{\boldsymbol{r}}\omega_Y^{(\psi_{\boldsymbol{r}})}(a,b)-\varphi_\ell'(b)$ and in the variance $\text{Var}_{\boldsymbol{r}}\,\omega_Y^{(\psi^{[\boldsymbol{r}]})}(a,b)$ all have a factor $J-1$ in the denominator. Our theoretical analysis thus proves that $\conceft$, using a larger dimensional space of TF representations, and subsequently averaging over the SST corresponding to random vectors in this larger dimensional space, leads to sharper estimates of the instantaneous frequencies for signals in $\mathcal{A}$ that are corrupted by noise. Even when $J$ is not large, our bounds show that $\conceft$ leads to a reduction in potential deviation of the tvPS from the itvPS.

The detailed estimates given in section ESM-3 are derived under the restrictive conditions listed at the start of this section for the signals and the wavelets used. However, as noted above, these conditions can be relaxed significantly (at the price of more intricate estimates). In practice, we observe similar behavior in our numerical examples even for more complex situations; in particular, the method can handle noise models that are much more challenging, as illustrated in the next section as well as by Figure \ref{fig:Introduction:Example2} in section \ref{Section:Intro}.

\section{Numerical Experiments}\label{Section:Conceftnumerical}

In this Section, we demonstrate the results of the $\conceft$ algorithm on examples; we also discuss different choices for some of the different parameters involved. The $\conceft$ Matlab code and the codes leading to the figures in this paper could be found in \url{https://sites.google.com/site/hautiengwu/home/download}.

The first choice to be made, when applying CWT- or STFT-based $\conceft$, concerns the family of orthonormal reference functions (wavelets or window functions) for the underlying wavelet or windowed Fourier transform. In both cases, we pick a family of eigenfunctions for a time-frequency localized operator designed for the CWT or STFT framework; as shown in \cite{Daubechies:1988,DaubPaul:1988} these can provide ``optimal'' localization within a restricted region of time-frequency space, where the size of the region depends solely on the number of functions used. More precisely, we use orthonormal Hermite functions for the STFT case \cite{Daubechies:1988,Xiao_Flandrin:2007} (see also Figure \ref{fig:Introduction:Example2}), and Morse wavelets for the CWT case \cite{DaubPaul:1988,Olhede_Walden:2002}. In both cases, the shape of the localization domain in TF plane is not completely fixed, but can be adjusted by varying some parameters; for details, see ESM. Once the family of orthonormal reference functions is fixed, we need to decide how many $\psi_j$, $j=1, \ldots,J$ we pick; this corresponds to choosing the size of the corresponding domain of concentration in the TF plane. Flexibility in the choices of shape and size of the TF localization domain make it possible to adapt $\conceft$, to some extent, to the family of signals under consideration. Finally, $\conceft$ also depends on the number $N$ of random projections chosen (see sections \ref{Section:Conceftalgorithm} and \ref{Section:Concefttheory}). In principle, the larger $N$, the closer the results are to the expected value of the random process, and the more we expect accidental correlations between reference function and the noise to cancel out in regions of the TF plane where the signal does not reside; in practice, increasing $N$ beyond a certain value does not appreciably improve the results. In what follows, we explore these different choices for the CWT case, on a simple family of challenging examples, with noise of different types (white Gaussian, Poisson and ARMA), and of different strengths. Results for the STFT case are similar; we will come back to them briefly below in subsection \ref{sec:STFT} as well as (in more detail) in the ESM. 

For our test data, we restrict ourselves to simulated signals only, so as to be able to quantify the deviation from the ``ground truth'', usually not available in real-life applications. ($\conceft$ results on concrete signals will appear elsewhere \cite{Lin_Wu:2015}.) On the other hand, we want to avoid parametric models, so as to be sufficiently general. Accordingly, we generate a class $\mathcal{C}$ of non-stationary data via a random process described below in subsection \ref{sec:data}; each realization provides us not only with a (simulated) clean signal, but also with the exact ``ground truth'' for the time-dependent instantaneous frequency and amplitude of the components of that signal. The same subsection also describes in detail three different noise models (white Gaussian, Poisson and ARMA(1,1)) for which the approach is tested. After applying $\conceft$ to signals in $\mathcal{C}$, we want to compare the $\conceft$ results with the optimal, ground truth TF representation; to quantify their (dis)similarity, we use an Optimal Transport (OT) distance, as described in subsection~\ref{sec:OT}. In subsection \ref{sec:param}, we discuss how choices of the parameters and of the number of orthogonal Morse wavelets impact the $\conceft$ results, for this family of examples; subsection~\ref{sec:N} illustrates the effect of the number $N$ of random projections. Finally, in subsection~\ref{sec:noise}, we explore the effect on CWT-based ConceFT of different noise levels, for each of the three noise types we consider; subsection \ref{sec:STFT} briefly discusses the STFT case. 

\subsection{Data simulation\label{sec:data}}
To generate a typical multi-component signal, we use smoothened Brownian path realizations to model the non-constant amplitudes and the instantaneous frequencies of the components; more precisely, if $W$ is the standard Brownian motion defined on $[0,\infty)$, then we define the {\it smoothened Brownian motion with bandwidth $B>0$} as $\Phi_{B}:=W\star K_{B}$, where $K_{B}$ is the Gaussian function with standard deviation $B>0$ and $\star$ denotes the convolution operator. Given $T>0$ and parameters $\zeta_1, \ldots \zeta_6 >0$, we then define the following family of random processes on $[0,T]$:
\[
\Psi_{[\zeta_1, \ldots \zeta_6]}(t) := \zeta_1 + \zeta_2\,t + 
\zeta_3 \frac{\Phi_{\zeta_4}(t)}{\|\Phi_{\zeta_4}\|_{L^\infty[0,T]}}
+ \zeta_5 \int_0^t \frac{\Phi_{\zeta_6}(s)}{\|\Phi_{\zeta_6}\|_{L^\infty[0,T]}}\ud s~.
\]
For the amplitude $A(t)$ of each IMT, we set $\zeta_2=\zeta_5=0$; every realization then varies smoothly between $\zeta_1$ and $\zeta_1+\zeta_3$. In the examples shown below and in the ESM, the signal consists of two components (i.e. $L=2$) on $[0,60]$; their two amplitudes are independent realizations of $\Psi_{[2,0,1,200,0,0}(t)]$. To simulate a phase function, we set $\zeta_1=\zeta_3=0$; $\Psi_{[0,\zeta_2,0,0,\zeta_5,\zeta_6]}(t)$ is then, appropriately, a monotonically increasing process. In the examples we consider, we take for $\varphi_1(t)$ a realization of $\Psi_{[0,10,0,0,6,400]}(t)$ for $t \in [0,60]$, and for $\varphi_2(t)$ a realization of $\Psi_{[0,2 \pi,0,0,2,300]}(t)$. Finally, we also constrain each component to ``live'' on only part of the interval, by setting
\[
s(t)=A_1(t)\cos(2\pi\varphi_1(t))\chi_{[18,\,60]}(t)+
A_2(t)\cos(2\pi\varphi_2(t))\chi_{[0,\,36]}(t) =: s_1(t) + s_2(t)\,,
\]
where $\chi_{[\tau_1,\tau_2]}$ is the indicator function of $[\tau_1,\tau_2]$; that is, $\chi_{[\tau_1,\tau_2]}(t)= 1$ if $\tau_1 \leq t \leq \tau_2$, $\chi_{[\tau_1,\tau_2]}(t)= 0$ otherwise. We shall denote the resulting class of two-component signals by $\mathcal{C}$. In our examples, signals in $\mathcal{C}$ are sampled uniformly at rate $160$Hz, corresponding to $9600$ samples. Figure~\ref{fig:data} plots $s(t)$ for one example $s \in \mathcal{C}$, as well as the instantaneous frequencies (IFs) of its two components, all restricted to the subinterval $[15,40] \subset [0,60]$.

\begin{figure}[htbp!]
\centering
\includegraphics[width=.95\textwidth]{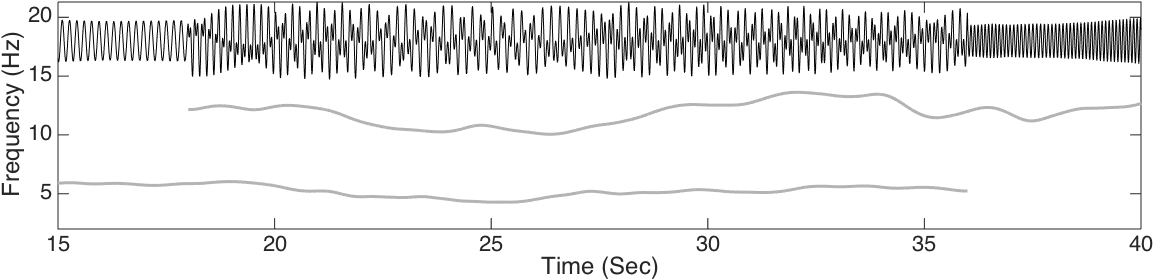}
\caption{The signal $s$ (in black) and the corresponding instantaneous frequencies (in gray) of the two components, restricted to the time interval $[15,40]$. 
\label{fig:data}}
\end{figure}

Note that the signal $s$ should not be viewed as a random process itself -- we use the random processes $\Psi_{[\zeta_1, \ldots \zeta_6]}$ as a means to generate signals consisting of several components for which the amplitudes and instantaneous frequencies are not easily expressed analytically, but we will not consider or compute expectations with respect to these processes -- once $s \in \mathcal{C}$ is generated, we consider it fixed when we apply $\conceft$ to it.  (In further subsections, we shall encounter other elements of $\mathcal{C}$.) 

To study the performance of $\conceft$ in the presence of noise, we add noise to $s(t)$, setting $Y(t_k)=s(t_k)+\sigma\xi(t_k)$, where $t_k$ is the $k$-th sampling time and $\xi$ is a stationary random process. We shall consider three different noise models; in each case we set the value of $\sigma$ so that the signal to noise ratio (SNR),
\[
\text{SNR}:=20\log\frac{\text{std}(s)}{\sigma\text{std}(\xi)}~,
\]
equals 0 dB. The three noise models we consider are Gaussian white noise, an auto-regressive-and-moving-average (ARMA) noise and Poisson noise.  For the ARMA case, we consider an ARMA$(1,1)$ model determined by autoregression polynomial $a(z)=0.5z+1$ and moving averaging polynomial $b(z)=-0.5z+1$; for the innovation process we use independent and identically distributed Student $t_4$ random variables. [Note that this ARMA$(1,1)$ noise is not white, because of the time dependence; in addition, the Student $t_4$ random variable has a ``fat-tailed'' distribution, resulting in possibly spiky realizations.]  For the Poisson noise, we pick the $\xi(t_k)$ to be independent and identically sampled from the Poisson distribution with parameter $\lambda=1$. Figure~\ref{fig:data_noise} plots a realization of $Y(t)=s(t)+\sigma\xi(t)$ for each of these three noise processes, restricted to the subinterval $[15,40]$.

\begin{figure}[htbp!]
\centering
\includegraphics[width=1\textwidth]{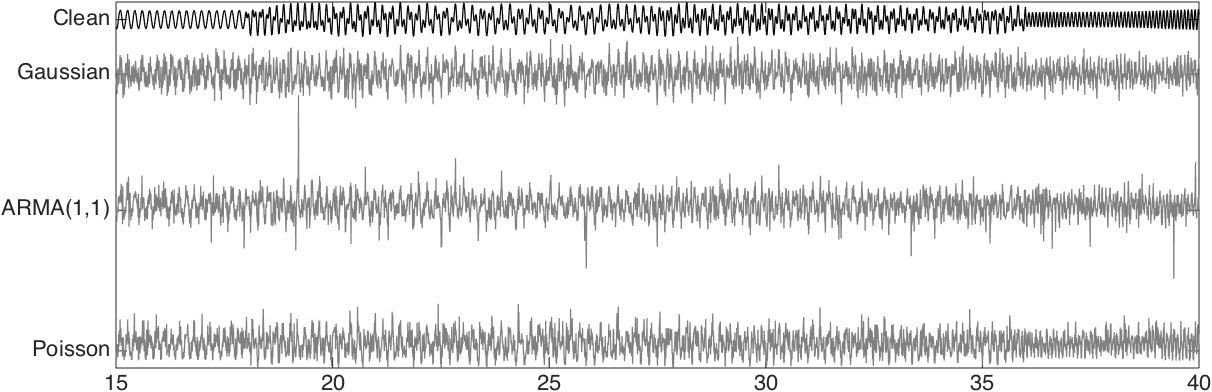}
\caption{The restrictions to $[15,40]$ of the clean signal $s$ (top) and of the noisy signal $Y=s+\sigma \xi$, where the added noise is Gaussian, ARMA(1,1), or Poisson noise (2nd row to bottom, in order); in each case $\sigma$ is picked so that the noisy signal has 0 dB SNR (signal to noise ratio). All signals are plotted at the same scale.\label{fig:data_noise}}
\end{figure}

\subsection{Performance evaluation\label{sec:OT}}

To evaluate the performance of $\conceft$, we propose comparing the time-varying power spectrum or tvPS (defined at the end of Section \ref{Section:Conceftalgorithm}) of the results of the $\conceft$ analysis of $Y$ with the {\it ideal time-varying power spectrum} (itvPS) of our simulated signal $s$, which can easily be defined explicitly (because our construction was designed accordingly) as follows:
 \[
\texttt{P}_{s}(t,\omega):= \sum_{k=1}^2 A_k^2(t) \,\delta_{\varphi'_k(t)}(\omega). \nonumber
 \]
In order to quantify the (dis)similarity between the $\conceft$-estimated tvPS $\widetilde{\texttt{P}}_Y$ and the itvPS $\texttt{P}_s$, we use the {\em Optimal Transport} distance (also called the Earth Mover distance). Because the principle of $\conceft$ is to ``reassign'' content in the TF plane, keeping the time-variable fixed (see Section \ref{Section:Conceftalgorithm}), we also keep $t$ fixed for the OT-distance. That is, we interpret, at each time $t$, $\widetilde{\texttt{P}}_{Y}(t,\omega)$ and $\texttt{P}_{s}(t,\omega)$ as (probability) distributions in $\omega$ and compute the OT-distance between them, which essentially measures how much one distribution needs to be ``deformed'' in order to coincide with the other; this is repeated for all $t$, and the average of the $t$-dependent individual OT-distances then indicates the quality of the estimator $\widetilde{\texttt{P}}_Y$ for $\texttt{P}_s$. 

The precise definition of (the discretized version of) the OT-distance we use is given in the ESM; Figure~\ref{fig:OT_demo} displays 4 examples in which two delta-measures localized on curves in the TF-plane (similar to the itvPS defined above) lie at similar OT-distances of each other -- although in each example the distance indicates a different type of ``distortion''. Together, these examples give an intuitive understanding of the way in which OT distances capture the difference between the TF distributions of interest to us here. 

\begin{figure}[htbp!]
\centering
\includegraphics[width=\textwidth]{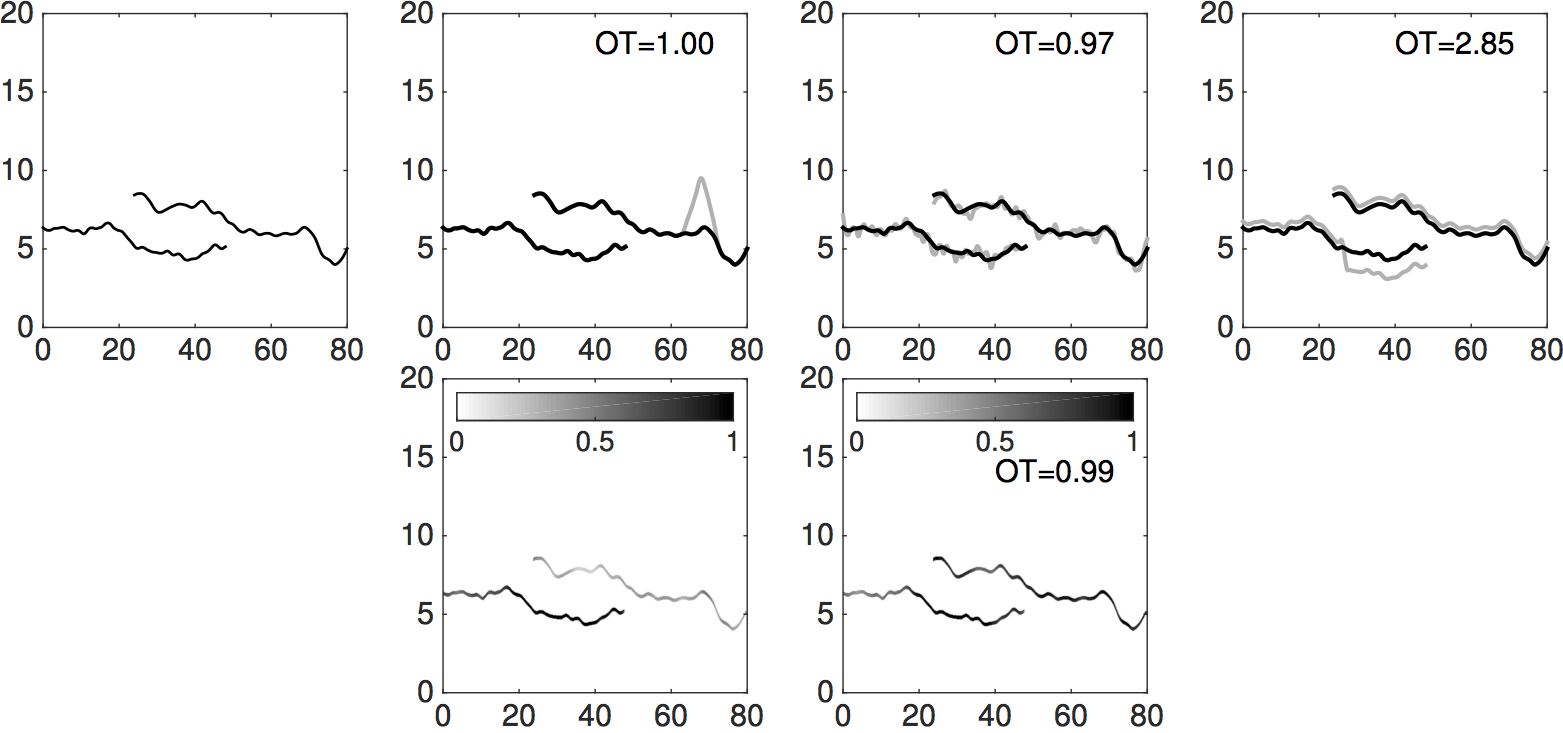}
\caption{Top row: left: the TF-localization of the ideal time varying power spectrum (itvPS) of a 2-component simulated signal $s_{a}$ (not showing the amplitude modulation (AM)). All other itvPS shown in the top row are for signals $s_{b}$, $s_{c}$ and $s_{d}$ that have a fairly small different OT-distance with respect to $s_{a}$; the two components have the same time-dependent amplitudes as for $s_{a}$, but the instantaneous frequency curves have been moved (in order from left to right) by a narrow bump (left), a random dither (middle) and a shift (right). Bottom row: illustration of amplitude change: left: the original itvPS of $s_{a}$ with the AM values indicated by gray scale level; right: an itvPS example with the same IF but different AM. In all the figures, the horizontal axis is time and the vertical axis is frequency. The image for each ``deformed'' itvPS indicates its OT-distance to the original itvPS (shown in the leftmost image on each row).\label{fig:OT_demo}}
\end{figure}

\subsection{Parameter selection\label{sec:param}}

As described in \cite{Olhede_Walden:2002}, generalizing the construction in \cite{DaubPaul:1988}, orthonormal families of Morse wavelets can be defined for different values of two parameters, $\beta$ and $\gamma$; different choices correspond to different shapes of the domain in the TF-plane on which they are mostly localized (see ESM). Once the values of $\beta$ and $\gamma$ are chosen, determining the family of $\psi_j$, one also needs to select $J$, the total number of orthonormal reference wavelets used in the $\conceft$ method. For signals in $\mathcal{C}$ (see \ref{sec:data}), we explored systematically a range of $(\beta,\gamma)$ pairs, as well as different values of $J$, to find the choice that, under different types of noise, with SNR of 0 dB, gave rise to the smallest OT-based distance (as described above) between the itvPS and the $\conceft$-estimated tvPS. Surprisingly, the optimal choice depended very little on the type of noise; the optimal values we found are $\beta=30$, $\gamma=9$ and $J=2$. (Detailed results are given in the ESM.)

\subsection{Effect of the number of random projections\label{sec:N}}

The $\conceft$ Algorithm averages the SST results computed with $N$ randomly picked reference wavelets (or windows, for STFT) from the linear span of the $\psi_j$, $j=1,\ldots,J$. It is expected that the concentration in the TF-plane observed with $\conceft$ kicks in only when $N$ is sufficiently large; on the other hand, the larger $N$, the more expensive the computation. To explore the trade-off, we applied $\conceft$ to the three noisy versions of the signal $s \in \mathcal{C}$ (see subsection \ref{sec:data}), with $N$ ranging from 1 to 200. In all cases, the $\conceft$ algorithm uses the optimal parameters as described in subsection \ref{sec:param}, i.e. it uses the first $2$ Morse wavelets with parameters $\beta=30,\,\gamma=9$. In this simulation, each $\conceft$ computation was repeated 300 times and the mean and standard deviation of the OT distances of the $\conceft$-tvPS to the itvPS were computed. Figure~\ref{fig:N} plots the results. For each of the three noise types, the graph of the average OT-distance shows an ``elbow'' shape, i.e. a regime in which the decrease is faster, as $N$ increases, followed by one in which the decrease is less marked. The elbow is located around $N=20$; the standard deviation is also quite small for this $N$. We accordingly decided to set $N=20$ in our further experiments. 

\begin{figure}[htb]
\centering
\includegraphics[width=\textwidth]{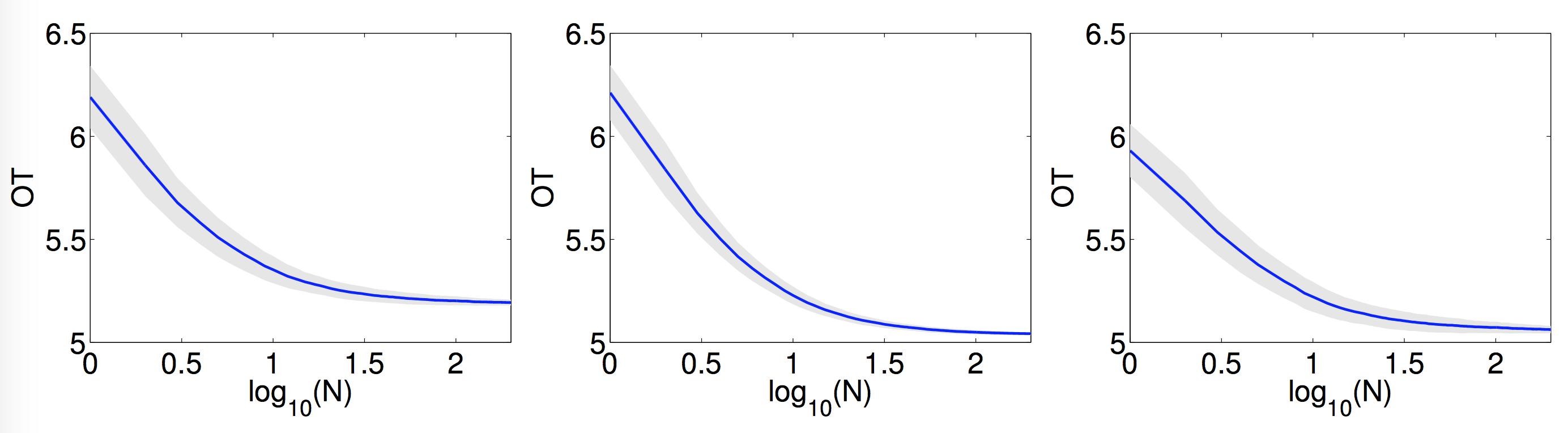}
\caption{The OT distance as a function of the number $N$ of random projections. The shaded band indicates the standard deviation of the OT distance at the corresponding number of random projections. The left column is for the first example and the right column for the second. From left to right, the noise types are Gaussian, ARMA(1,1) and Poisson respectively. For all three experiments, $\beta=30,\,\gamma=9$, and the first two Morse wavelets are used. \label{fig:N}}
\end{figure}

\subsection{$\conceft$ results for noisy signals\label{sec:noise}}
We now show the result of using $\conceft$ with the calibrated parameter choices. We illustrate the performance of $\conceft$ on signals of the simulation class $\mathcal{C}$ (see subsection \ref{sec:data}), for a range of SNR, as well as on deterministic signals. 

\begin{figure}[htbp!]
\centering
\includegraphics[width=\textwidth]{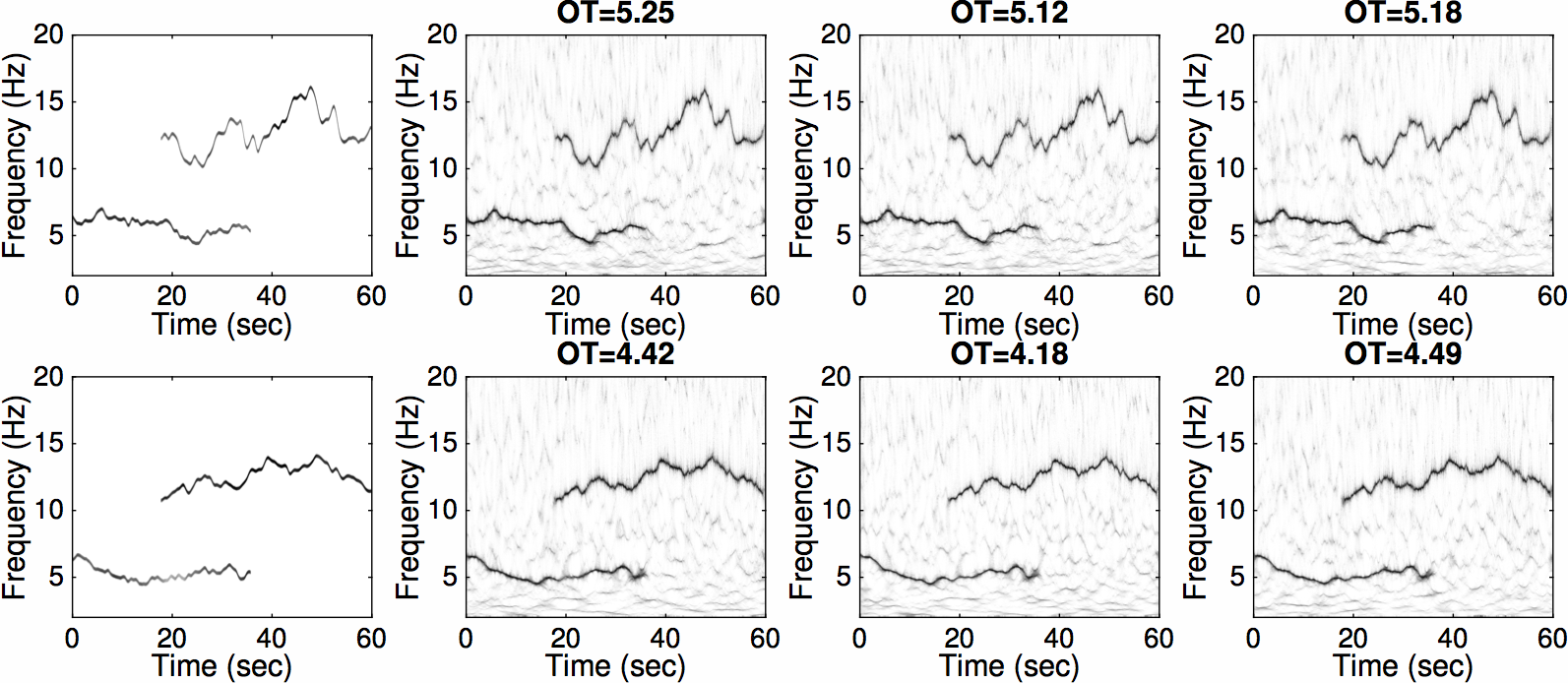}
\caption{First row: results for the signal $s$; second row: results for a new example $s^{\ast}$.  Left to right: ideal time-varying TF power spectrum (itvPS) for the clean signal, followed by results of $\conceft$ with Morse wavelets after (in order) Gaussian, ARMA(1,1) or Poisson noise was added, with SNR of 0 dB. Clearly, even for a signal-to-noise ratio is as low as $0$ dB, the results approximate the truth with high precision. For each of the tvPS panels, the header gives the OT distance to the corresponding itvPS.  \label{fig:nsType}}
\end{figure}

As a warm-up, we start with the signal $s$ seen before. The top row of Figure \ref{fig:nsType} plots the tvPS $\widetilde{\texttt{P}}_Y$ of the three noisy versions of $s$ next to the itvPS $\texttt{P}_s$. To compress the dynamical range of the tvPS plots, we carry out the following procedure. We first normalize the discretized version $\boldsymbol{\widetilde{\mathbb{P}}}_Y\in\RR^{m\times n}$ of $\widetilde{\texttt{P}}_Y$ (where $m$ and $n$ stand for the number of discrete frequencies and the number of time samples, respectively) by multiplying it by a constant so that {\color{black}the total weight of all entries equals the same number for all cases -- i.e., for some $\theta>0$ (to be picked -- see below), $\frac{1}{nm}\sum_{k=1}^m\sum_{l=1}^n \left(\boldsymbol{\widetilde{\mathbb{P}}}_Y\right)_{k,l} = \theta$.} 
We then plot a gray-scale visualization of $\boldsymbol{R}\in\RR^{m\times n}$ rather than the (normalized) $\boldsymbol{\widetilde{\mathbb{P}}}_Y\in\RR^{m\times n}$ itself, where $\boldsymbol{R}_{k,l}:=\log(1+\min\{\boldsymbol{\widetilde{\mathbb{P}}}_{k,l},q\})$, $k=1,\ldots,m$, $l=1,\ldots,n$ and $q$ is a (very high) cut-off to downplay the effect of far-off outliers. We choose $q$ to be the same for all three tvPS, so that comparable gray levels on the different tvPS panels indicate comparable values of $\boldsymbol{R}$ (see Section 4f in the ESM for a more extensive discussion of choosing $q$ and gray-scale plotting of tvPS). 
{\color{black}For the figures, we choose $\theta=5$ and $q=5.718$; this value for $q$ is the
minimum of the $99.8\%$ quantiles of the different tvPSs.}
The second row of Figure \ref{fig:nsType} gives the results for $s^{\ast}$, a signal of the simulation class $\mathcal{C}$ that was not used (in contrast to $s$) to calibrate parameters of $\conceft$. The results are similarly highly accurate.

Next, we study the effect on the $\conceft$ performance of the noise level, as quantified by SNR. To this end, we revisit the analysis of the signal $s^{\ast}$ (and $s$ in the ESM). For each signal, each type of noise (Gaussian, ARMA(1,1) or Poisson) and each SNR considered (SNR= $x$ dB, where $x \in \{-7,-6,\ldots,6,7\}$), we considered 20 independent realizations of the noise process; for each of the resulting noisy signals we carried out the $\conceft$ analysis and computed the OT-distance of the tvPS to the itvPS of the clean signal; we then computed the mean and the standard deviation for each. The results are shown in  Figure~\ref{fig:sigma}. The same figure also compares the $\conceft$ results with those of simple SST (using  either the first Morse wavelet with parameters $\beta=30,\,\gamma=9$ as reference wavelet, or {\em one} random linear combination of the two first Morse wavelets) and of multi-taper SST (denoted as \textsf{orgMT}), using the same $\psi_j$ as $\conceft$. For each of these alternate methods, we likewise computed the mean OT-distance of the tvPS to the itvPS for 20 noise realizations. It is striking that the $\conceft$ method outperforms the other methods in all cases. 

\begin{figure}[htbp!]
\centering
\includegraphics[width=\textwidth]{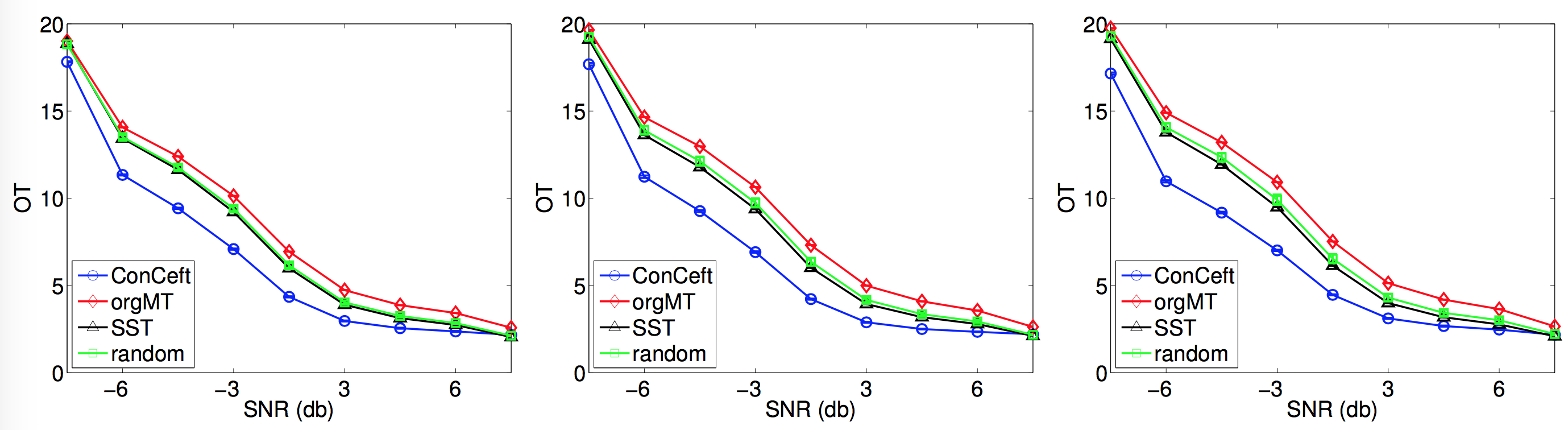}
\caption{OT distance of $\conceft$ tvPS results against signal to noise ratio (SNR) of the signal  $s^{\ast}(t)$, and comparison with standard SST and standard multi-taper SST (see text). Noise type (left to right): Gaussian, ARMA(1,1), and Poisson. The standard deviation is smaller, at the scale of this figure, than the height of the markers, and has not been plotted.\label{fig:sigma}}
\end{figure}

Finally, to address possible concerns that the randomness in the generation and plots of $\varphi'(t)$ and $A(t)$ somehow ``help'' $\conceft$ in these estimations, we show in Figure~\ref{fig:other_signal} the results for yet another signal, which (in contrast to $s$ and $s^{\ast}$) is completely deterministic; it consists of 3 components, each given by an explicit, analytic formula (again for $t \in [0,60]$):
\begin{align*}
s^{\circ}(t)=& \chi_{[10,48]}(t)\,\left(1+0.3\cos(\pi(t-10)/20)^2\right)\,\cos \left(\pi/3+5t+t^2/50\right)\\ &~~ + \left(0.4+0.9\sin(\pi t/60)^2 \right)\, \cos \left( 12t + \sin( \pi t /6 )\right) + 1.2\chi_{[15,60]}(t)\, \cos \left( 17 t + (t-35)^3/800  \right)~. 
\end{align*}
Figure~\ref{fig:other_signal} shows that the results are of a quality similar to those in Figure \ref{fig:nsType}.

\begin{figure}[htbp!]
\includegraphics[width=\textwidth]{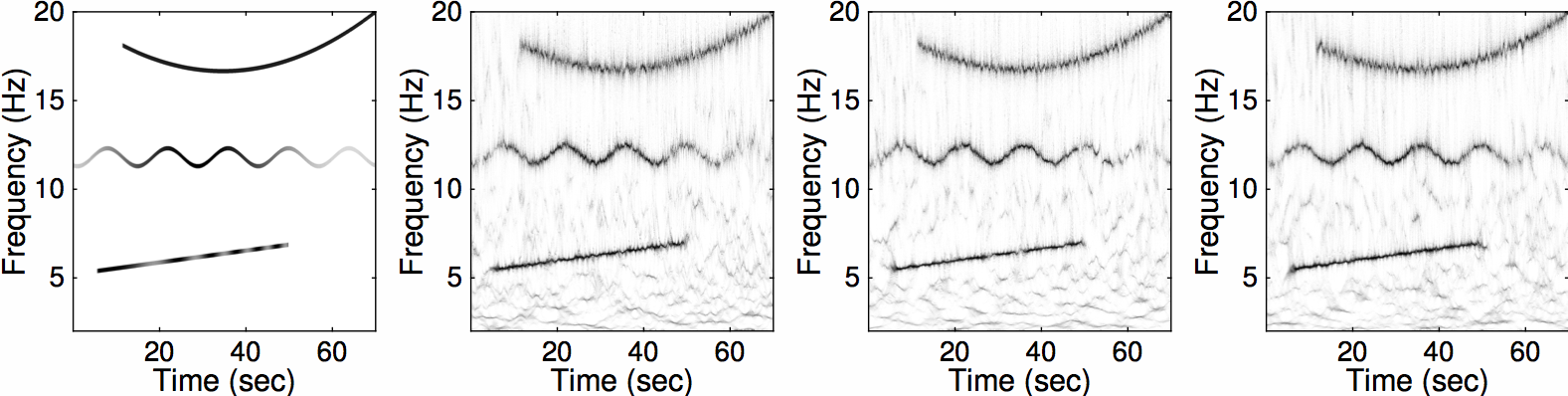}
\caption{Results for the three-component deterministic signal $s^{\circ}$. Left: ideal time-varying TF power spectrum (itvPS) for the clean signal, followed by results of $\conceft$ with Morse wavelets after (in order) Gaussian, ARMA(1,1) or Poisson noise was added, with SNR of 0 dB.    \label{fig:other_signal}}
\end{figure}

\subsection{$\conceft$ with STFT\label{sec:STFT}}

As described earlier, the $\conceft$ approach can be carried out for STFT-based SST as well as for CWT-based SST. Figure \ref{fig:Introduction:Example2} in section \ref{Section:Intro} already showed the results of STFT-$\conceft$ on one example. Other examples are shown in the ESM, together with values of the OT-distance of the STFT-$\conceft$ estimated tvPS to the itvPS. In these experiments, as in Figure \ref{fig:Introduction:Example2}, the reference windows are chosen to be Hermite functions; 20 random projections are used to compute the $\conceft$ averages. Although STFT-based $\conceft$ achieves better OT-distance with respect to the ground truth than STFT-based multi-taper SST, and also achieves a better reduction of ``background noise'' (i.e. the structures in zones away from itvPS concentration, due to fortuitous correlations between the noise and the overcomplete frame of TF reference functions; see the description in section \ref{Section:Intro}), the performance of STFT-based $\conceft$ is not quite as impressive, on the class $\mathcal{C}$, as CWT-based $\conceft$. We provide some discussion in the ESM.

\section{Conclusion}\label{Section:Discussion}
We consider signals that are the linear combination of a small number of
``intrinsic-mode functions'', each of which can be reasonably viewed as
an oscillatory function with well-defined but time-varying amplitude and
``instantaneous frequency''. We have introduced a new approach, called
ConceFT, to determine
the time-frequency representation of such signals, combining
multi-taper estimation
ideas and averaging over random projections with synchrosqueezing. Theoretical analysis shows
that this leads to improved estimation of the time-varying characteristics
of the signals of interest; numerical results confirm the theoretical
promise, even when the signals are corrupted by significant
and challenging noise.

We also introduced two tools to evaluate the effectiveness of this method (or other
similar methods), which may be of interest in their own right to others working in the TF field.
On the one hand, we
introduced a class of explicit, easy to construct signals with explicit
time-varying characteristics,
even though the signals themselves are not given by explicit formulas;
the explicit time-varying amplitude and instantaneous frequency give a
``ground truth'' with which estimations can be compared.
On the other hand, we introduced a distance between time-frequency
representations that can be useful in comparing results
obtained by different methods, by computing for each the distance to the
``ground truth'' time-frequency representation.

\bibliographystyle{plain}
\bibliography{TFanalysis}

\newpage

\centerline{\bf {\Large Electronic Supplementary Materials for }}
\centerline{\bf {\Large``ConceFT: Concentration of frequency and time}}
\centerline{\bf {\Large  via a multi-tapered synchrosqueezing transform''}}

\setcounter{equation}{0}
\setcounter{figure}{0}
\setcounter{table}{0}
\renewcommand{\thefigure}{S.\arabic{figure}}
\renewcommand{\thetable}{S.\arabic{table}}
\renewcommand{\theequation}{S.\arabic{equation}}
\renewcommand{\thesection}{S.\arabic{section}}
\renewcommand{\thetheorem}{S.\arabic{theorem}}
\renewcommand{\thelemma}{S.\arabic{lemma}}

\section*{ESM-1. Introduction}
These are the {\em Electronic Supplementary Materials} for the paper {\sc ConceFT: Concentration of frequency and time via a multi-tapered synchrosqueezing transform}. They contain, in particular, precise mathematical definitions, theorem statements and proofs that complement the more general exposition in the main body of the paper, as well as details about the numerical examples and additional examples. For the convenience of the reader, the organization into sections follows that of the paper; for instance, material in Section ESM-3 complements Section 3 in the main paper.

\section*{ESM-2. The ConceFT Algorithm: Several Remarks}
(A) As described in Section 2 in the main paper, the SST-steps in ConceFT involve the computation of a partial derivative, $\partial_b W_f^{(\psi)}(a,b)$ with respect to the localization parameter $b$ of $W_f^{(\psi)}(a,b)$. In practice, one has $W_f^{(\psi)}(a,b)$ only for discrete (as opposed to continuous) values of $a$ and $b$, and partial differentiation would be approximated by a differentiating scheme. This can cause stability issues when $f$ is noisy. Using the definition of $W_f^{(\psi)}(a,b)$ as the inner product of $f$ with $|a|^{-1/2}\psi\left(\frac{\cdot-b}{a}\right)$, one can compute $\partial_bW_f(a,b)$ via the wavelet transform of $f$ with respect to the wavelet $\psi'$ using $\partial_bW_f^{(\psi)}(a,b)=-W_f^{(\psi')}(a,b)$, typically this makes the computation more stable than simple numerical differencing. 

(B) The ConceFT algorithm consists in taking the average of many nonlinear SST estimates of the tvPS,
each of which results from a wavelet transform with respect to a randomly picked reference wavelet;
for each individual transform the corresponding reassignment is computed and carried out to find
that individual SST. An alternative to
the individual SSTs would be to define one ``master'' reassignment rule, as follows. 
From the collection of $W_f^{(\psi_j)}$, $j=1,\ldots,J$, 
we could estimate
$\Omega^{(\Gamma)}(a,b)$ as the value of $\xi$ for which the vector 
\[
\boldsymbol{w}(a,b)=[W^{(1)}_f(a,b), \partial_bW^{(1)}_f,\ldots,W^{(J)}_f(a,b), \partial_bW^{(J)}_f(a,b)]\in \CC^{2J}
\]
is most ``aligned'' with the vector $u(a,b,\xi)$ 
\[
\boldsymbol{u}(a,b,\xi)=[W^{(1)}_{e^{i2\pi\xi t}}(a,b), \partial_bW^{(1)}_{e^{i2\pi\xi t}},\ldots,W^{(J)}_{e^{i2\pi\xi t}}(a,b), \partial_bW^{(J)}_{e^{i2\pi\xi t}}(a,b)]\in \CC^{2J}.
\]
In other words, the reassignment rule would become
\[
\Omega(a,b):=\argmax_{\xi}\frac{|\langle \boldsymbol{w}(a,b),\boldsymbol{u}(a,b,\xi)\rangle|}{\|\boldsymbol{w}(a,b)\|\|\boldsymbol{u}(a,b,\xi)\|}\,.
\]
Although numerical experiments have shown this to be an interesting approach as well, we shall not
pursue this in this paper.

(C) In most of our examples and figures, we concentrate on visualizing the location in the TF plane of the curves characterizing the different IMT components of the signals considered. However, we can also use the tvPS constructed by $\conceft$ to estimate the different amplitudes, as follows. Each $\conceft$ tvPS is the average of many SSTs constructed in such a way that the integral (sum, in practice) over $\xi$, on an interval around $\phi_l'(t)$, approximates $A_l(t)\cos(2\pi\phi_l(t))$ (see \cite{Daubechies_Lu_Wu:2011,Chen_Cheng_Wu:2014}). It follows that one can use the $\conceft$ representation to first identify the $\phi_l'(t)$ for all $t$ (which can be done more stably with $\conceft$, for large noise, than with simple SST of MTSST), and then integrate $\tilde{S}_Y(t,\xi)$ with respect to $\xi$ in an appropriate interval around $\phi'_l(t)$, to recover $A_l(t)$.

\section*{ESM-3. Theoretical Results: Mathematical statements and proofs}\label{ESM-3}

The following is the mathematically precise definition of an {\it intrinsic-mode type} (IMT) function:
\begin{defn} 
Given $\epsilon$, $c_1$ and $c_2$ 
satisfying  $0<\epsilon\ll 1$, $0<c_1\leq c_2<\infty$,
a function $F(t)$ is said to be of type 
$\mathcal{A}^{c_1,c_2}_{\epsilon}$ if it can be written as
\[
F(t) = A(t)\cos(2\pi\varphi(t)),
\]
where
\begin{align}\label{definition:adaptiveHarmonicSingle}
\left\{\begin{array}{l}\vspace{.2cm}
A\in C^1(\RR)\cap L^\infty(\RR),\quad\varphi\in C^2(\RR),\\ \vspace{.2cm}
\inf_{t\in\RR} A(t)>c_1,\quad \inf_{t\in\RR}\varphi'(t)>c_1,\\ \vspace{.2cm}
\sup_{t\in\RR} A(t)\leq c_2,\quad\sup_{t\in\RR}\varphi'(t)\leq c_2,\\ \vspace{.2cm}
|A'(t)|\leq \epsilon \varphi'_\ell(t),\quad |\varphi''(t)|\leq \epsilon\varphi'(t)\quad\mbox{ for all }t\in\RR,
\end{array}\right.
\end{align}
\end{defn}

To model the oscillatory functions with different oscillatory modes, we also consider superpositions of IMT functions:

\begin{defn} 
Given $\epsilon$, $c_1$, $c_2$ and $d$ satisfying  $0<\epsilon\ll 1$, $0<c_1\leq c_2<\infty$, $0<d<1$, a function $G(t)$ is said to be of type $\mathcal{A}^{c_1,c_2}_{\epsilon,d}$ if it can be written as
\begin{equation}\label{decomp1}
G(t)=\sum_{\ell=1}^L F_\ell(t) 
\sum_{\ell=1}^L A_\ell(t)\cos(2\pi\varphi_\ell(t)),
\end{equation}
where each  
$F_\ell=A_\ell(\cdot)\cos(2\pi\varphi_\ell(\cdot))$ is of type
$\mathcal{A}^{c_1,c_2}_{\epsilon}$ and
\begin{align}\label{definition:adaptiveHarmonicMultiple}
\varphi_{\ell+1}'(t)-\varphi'_\ell(t)>d(\varphi_{\ell+1}'(t)+\varphi_{\ell}'(t))
\end{align}
for all $\ell=1,\ldots,L-1$. 
\end{defn}

Finally, we also consider the additive white Gaussian noise. Denote $\mathcal{S}$ to be the Schwartz function space. Our model for the observed signal $Y(t)$ is thus
\begin{align}\label{decompAdaptive}
Y(t) = \sum_{\ell=1}^L F_\ell(t) + \sigma\Phi(t) = \sum_{\ell=1}^L A_\ell(t)\cos(2\pi\varphi_\ell(t))+\sigma\Phi(t),
\end{align}
where $G=\sum_{\ell=1}^L F_\ell$ is of type $\mathcal{A}^{c_1,c_2}_{\epsilon,d}$, $\Phi$ is a Gaussian white noise so that the standard deviation of $\Psi(\psi)$ is $1$ for all $\psi\in\mathcal{S}$ with norm $1$, and $\sigma>0$ is the noise level; $Y$ is a generalized random process, since by definition $\sum_{\ell=1}^K A_\ell(t)\cos(2\pi\varphi_\ell(t))$ is a tempered distribution. 

The $J$ reference wavelets $\psi_1,\ldots,\psi_J$ are orthonormal, that is, $\int\psi_i(x)\overline{\psi_j(x)}\ud x= \delta_{i,j}$, where $\delta_{i,j}$ is the Kronecker delta. For simplicity we assume that the $\psi_j$ all have fast decay, that their Fourier transforms $\widehat{\psi_j}$ are real functions with compact support, and $\mbox{supp}\widehat{\psi_j}\subset [1-\Delta_j,1+\Delta_j]$, where $0<\Delta_j<1$. 

We shall consider the continuous wavelet transforms of $Y$ with respect to the $\psi_j$, and apply synchrosqueezing to them. For the $F_\ell$-components of $Y$ we refer the reader to the detailed analysis in \cite{Daubechies_Lu_Wu:2011, Chen_Cheng_Wu:2014}. In particular, we introduce the sets $Z^{(j)}_\ell(b)=\left[\frac{1-\Delta_j}{\varphi_\ell'(b)}, \frac{1+\Delta_j}{\varphi_\ell'(b)}\right]$. If each $\Delta_j$ satisfies $\Delta_j\leq \frac{d}{1+d}$ for $j=1,\ldots,J$ (which we shall assume for the remainder of this discussion), then one finds that, by the conditions on $G$, the sets $Z^{(j)}_\ell(b)$ are disjoint. Moreover, the CWT $W^{(\psi_j)}_G(a,b)$ is small except for those pairs $(a,b)$ where $a \in Z^{(j)}_\ell(b)$ for some $\ell$, and in that case $\frac{-i\partial_b W^{(\psi_j)}_{G}(a,b)}{2\pi W^{(\psi_j)}_{G}(a,b)}$ is close to $\varphi_\ell'(b)$. (See Theorem 3.3 in \cite{Daubechies_Lu_Wu:2011}.)

It will be convenient to use $\underline{\Delta}:=\min_{j=1}^J(\Delta_j)$, $\overline{\Delta}:=\max_{j=1}^J(\Delta_j)$, and $Z_\ell(b)=\left[\frac{1-\underline{\Delta}}{\varphi_\ell'(b)}, \frac{1+\underline{\Delta}}{\varphi_\ell'(b)}\right]$. Clearly $Z_\ell(b)=\cap_{j=1}^J Z^{(j)}_{\ell}(b)$.

As shown by the analysis in 
\cite{Daubechies_Lu_Wu:2011, Chen_Cheng_Wu:2014}, we have
\[
W^{(\psi_j)}_G(a,b)=\left\{
\begin{array}{ll}
e^{i2\pi\varphi_\ell(b)}Q_{j,\ell}(a,b)+\epsilon_j(a,b)&\mbox{ when }a\in Z^{(j)}_\ell(b) \,\mbox{ for some }\, \ell=1,\ldots,L\\
\epsilon_j(a,b)&\mbox{ otherwise},
\end{array}
\right.
\]
where
\begin{equation}\label{Expansion:Qk}
Q_{j,\ell}(a,b)=A_\ell(b)\sqrt{a}\widehat{\psi_j}(a\varphi'_\ell(b))\in \RR
\end{equation}
and $\epsilon_j(a,b)$ is of order 
$\oeps=\epsilon^{1/3}$ for all $j=1,\ldots,J$. 

Adding also the noise, we have thus
\[
W^{(\psi_j)}_Y(a,b)=\sum_{\ell=1}^L e^{i2\pi\varphi_\ell(b)}Q_{j,\ell}(a,b)
\chi_{Z^{(j)}_\ell}(b)+\epsilon_j(a,b)+\sigma\Phi(\psi_{j}^{(a,b)}),
\]
where 
$\psi_{j}^{(a,b)}(t):=\frac{1}{\sqrt{a}}\psi_j\left(\frac{t-b}{a}\right)$ 
and $\chi_{Z^{(j)}_\ell(b)}$ is the indicator function of the set
$Z^{(j)}_\ell(b)$.

To simplify further notation, we shall use boldface for $J$-dimensional ``vector'' quantities; for instance we denote $\boldsymbol{\psi}:=[\psi_1,\psi_2,\ldots, \psi_J]^\intercal\in\oplus^J \mathcal{S}$ and $\psi^{[\boldsymbol{r}]}:=\boldsymbol{r}^\intercal\boldsymbol{\psi}$, where $\boldsymbol{r}\in S^{J-1}=\{\boldsymbol{v}\in\RR^J\,;\, \|\boldsymbol{v}\|^2 = \sum_{j=1}^J v_j^2 = 1\}$. Clearly $\psi^{[\boldsymbol{r}]}$ is also a Schwartz function, with $\mbox{supp}\left(\widehat{\psi^{[\boldsymbol{r}]}}\right)\subset [1-\overline{\Delta},1+\overline{\Delta}]$. We similarly introduce $\boldsymbol{\epsilon}(a,b):= [\epsilon_1(a,b),\ldots, \epsilon_J(a,b)]^\intercal$ (a vector with norm of order $\oeps$), $\boldsymbol{Q_\ell}(a,b):=[Q_{1,\ell}(a,b),\ldots,Q_{j,\ell}(a,b)]^\intercal$. Note that all the entries of the vectors $\boldsymbol{Q}_\ell(a,b)$ are real; this will be important for our estimates below. $\boldsymbol{\Phi}(a,b):=[\Phi(\psi_{1}^{(a,b)},\ldots,\Phi(\psi_{J}^{(a,b)})]^\intercal$. $\boldsymbol{\Phi}(a,b)$ is a complex Gaussian random vector \cite{Gallager:2008}, of which the following Lemma gives some basic properties: 

\begin{lemma}\label{Lemma:Phi}
For all $a>0$ and $b\in\RR$, $\boldsymbol{\Phi}(a,b)$ is a complex Gaussian 
random vector with mean $[0,\ldots,0]^\intercal\in\RR^J$, 
for which the covariance matrix and the relation matrix 
both equal $I_{J\times J}$. 
Thus, for all $\boldsymbol{v}\in \RR^{J}$, 
$\boldsymbol{v}^\intercal\boldsymbol{\Phi}(a,b)$ is a complex Gaussian random variable with mean $0$ and variance $\|\boldsymbol{v}\|^2$.
\end{lemma}
\begin{proof}
Fix $a>0$ and $b\in\RR$. Since $\Phi$ is a Gaussian white noise and $\psi$ is a complex Schwartz function, it follows that for $j=1,\ldots,J$, $\Phi(\psi_{j}^{(a,b)})$ is a complex Gaussian random variable \cite{Gelfand:1964}. By definition, its mean is $0$ and its variance is
\begin{align}
\text{Var}(\Phi(\psi_{j}^{(a,b)}))&\,=\mathbb{E}|\Phi(\psi_{j}^{(a,b)})|^2 =\int \left|\widehat{\psi_{j}^{(a,0)}}(\xi)\right|^2\ud \xi =\|\widehat{\psi_j}\|_{L^2(\RR)}^2=1.\label{proof:thm:relationship_varX:1}
\end{align}
It is clear that the variance of $\Phi(\psi_{j}^{(a,b)})$ is independent of the scale $a$. Since $\psi_i$ and $\psi_j$ are orthogonal if $i \neq j$, a similar calculation shows that $\text{Cov}(\Phi(\psi_{i}^{(a,b)}),\Phi(\psi_{j}^{(a,b)}))=\delta_{i,j}$. Since we assume that $\widehat{\psi_j}$ is real for all $j=1,\ldots,J$, the relation matrix of $\boldsymbol{\Phi}(a,b)$ equals the covariance matrix, and is thus $I_{J\times J}$ as well. It then easily follows that $\boldsymbol{v}^\intercal\boldsymbol{\Phi}(a,b)$ is a complex Gaussian random variable with mean $0$ and variance $\|\boldsymbol{v}\|^2$.
\end{proof}

Because the CWT is (anti)linear in the wavelet with respect to which it is computed,
the CWT of $Y$ with respect to $\psi^{[\boldsymbol{r}]}$ is given by 
\begin{equation}
W^{(\psi^{[\boldsymbol{r}]})}_Y(a,b)=
\sum_{\ell=1}^L \sum_{j=1}^J r_j
e^{i 2\pi \varphi_j(b)} Q_{j,\ell}(a,b) \chi_{Z^{(j)}_e\\}(a,b)
+\boldsymbol{r}^\intercal\left[\boldsymbol{\epsilon}(a,b)
+\sigma\boldsymbol{\Phi}(a,b)\right].
\end{equation}

The analysis in \cite{Daubechies_Lu_Wu:2011, Chen_Cheng_Wu:2014} also shows that
\[
\partial_bW^{(\psi_j)}_G(a,b)=\left\{
\begin{array}{ll}
i 2\pi \left(\varphi_\ell'(b)e^{i 2\pi \varphi_\ell(b)} Q_{j,\ell}(a,b)
+\widetilde{\epsilon}_j(a,b)\right)&\mbox{ when }a\in Z^{(j)}_\ell(b) \,\mbox{ for some }\, \ell\\
i 2\pi \widetilde{\epsilon}_j(a,b)&\mbox{ otherwise},
\end{array}
\right.
\]
where $\widetilde{\epsilon}_j(a,b)$ is of order $\oeps$ for all $j=1,\ldots,J$.
We thus obtain
\[
-i\partial_bW^{(\psi^{[\boldsymbol{r}]})}_Y(a,b)
=\, 2\pi \sum_{\ell=1}^L \sum_{j=1}^R
r_j \varphi_j'(b) e^{i 2\pi \varphi_j(b)} Q_{j,\ell}(a,b) \chi_{Z^{(j)}_\ell}(a,b)
 2\pi \boldsymbol{r}^\intercal\left[\widetilde{\boldsymbol{\epsilon}}(a,b)
+\sigma\widetilde{\boldsymbol{\Phi}}(a,b)\right],
\]
where $\widetilde{\boldsymbol{\epsilon}}(a,b)=
[\widetilde{\epsilon}_1(a,b),\ldots,\widetilde{\epsilon}_J(a,b)]^\intercal$, and
$\widetilde{\boldsymbol{\Phi}}(a,b)=(2 \pi)^{-1}[\Phi(i(\psi_1^{(a,b)})'),\ldots,
\Phi(i(\psi_J^{(a,b)})')]^\intercal$. 
Here $\widetilde{\boldsymbol{\epsilon}}(a,b)$ is again a $J$-dim random vector 
with norm of order $\overline{\epsilon}$. The following lemma gives
some basic properties of the  complex random vector 
$\widetilde{\boldsymbol{\Phi}}(a,b)$:
\begin{lemma}\label{Lemma:Phip}
For all $a>0$ and $b\in\RR$, 
$\widetilde{\boldsymbol{\Phi}}(a,b)$ is a complex Gaussian random vector with 
mean $[0,\ldots,0]^\intercal\in\RR^J$, for which the covariance matrix and
the relation matrix both equal
$\diag[\|\widehat{\psi'_1}\|^2,\ldots,\|\widehat{\psi'_J}\|^2]/(2\pi a)^2\in\RR^{J\times J}$. 
Thus, for all $\boldsymbol{v}\in \RR^{p}$, $\boldsymbol{v}^\intercal
\widetilde{\boldsymbol{\Phi}}(a,b)$ is a complex Gaussian random variable 
with mean $0$ and variance 
$\sum_{j=1}^J \boldsymbol{v}_j^2\|\widehat{\psi'_j}\|^2/(2 \pi a)^2$.
\end{lemma}
\begin{proof}
The proof is the same as that of Lemma \ref{Lemma:Phi}, except 
for the following slight difference:
\begin{align}
\text{Var}(\Phi(i(\psi_j^{(a,b)})')&\,
=\mathbb{E}|\Phi((\psi_j^{(a,b)})')|^2
=\|(\psi_j^{(a,b)})'\|_{L^2(\RR)}
= \|\psi_j'\|_{L^2(\RR)}^2/a^2.\nonumber
\end{align}
\end{proof}
As a result, when $a\in Z_\ell(b)$, 
the reassignment rule, $\omega_Y^{(\psi^{[\boldsymbol{r}]})}(a,b)$, becomes
\begin{align*}
\omega_Y^{(\psi^{[\boldsymbol{r}]})}(a,b)
=\,&\frac{-i\partial_b W^{(\psi^{[\boldsymbol{r}]})}_G(a,b)
+2\pi\boldsymbol{r}^\intercal\sigma\widetilde{\boldsymbol{\Phi}}(a,b)}
{2\pi \left(W^{(\psi^{[\boldsymbol{r}]})}_G(a,b)
+\sigma \boldsymbol{r}^\intercal\boldsymbol{\Phi}(a,b)\right)}\\
=\,&\frac{\boldsymbol{r}^\intercal
\left[\varphi_\ell'(b)e^{i2\pi\varphi_\ell(b)}\boldsymbol{Q}_\ell(a,b)
+\widetilde{\boldsymbol{\epsilon}}(a,b)
+\sigma\widetilde{\boldsymbol{\Phi}}(a,b)\right]}
{\boldsymbol{r}^\intercal\left[e^{i 2\pi \varphi_\ell(b)}\boldsymbol{Q}_\ell(a,b)
+\boldsymbol{\epsilon}(a,b)+\sigma\boldsymbol{\Phi}(a,b)\right]},
\end{align*}

It follows from Lemma \ref{Lemma:Phi} and Lemma \ref{Lemma:Phip} that $\omega_Y^{(\psi^{[\boldsymbol{r}]})}(a,b)$ is a ratio random variable of two independent complex Gaussian random variables with non-zero means.

Note that we are implicitly assuming here that the denominator in the fraction for $\omega_Y^{(\psi^{[\boldsymbol{r}]})}(a,b)$ is not too small (see Section 2 in the main paper). In what follows, we shall make this explicit: we shall always assume that 
$$
2 \pi|W^{(\psi^{[\boldsymbol{r}]})}_Y(a,b)|
\,\left(\,= \left|\boldsymbol{r}^\intercal\left[e^{i 2\pi \varphi_\ell(b)}\boldsymbol{Q}_\ell(a,b)
+\boldsymbol{\epsilon}(a,b)+\sigma\boldsymbol{\Phi}(a,b)\right] \right|
\,\,\mbox{ if }\,a\in Z_\ell(b)\,\right)
$$ 
exceeds the value $2 \kappa$, where the value of $\kappa$ can be set (according to the signal characteristics and noise level). At the same time, we shall assume that $\overline{\epsilon}$ and $\sigma$ are sufficiently small that
$$
\mathbb{E}\left(
\|\boldsymbol{\epsilon}(a,b)+\sigma \boldsymbol{\Phi}(a,b)\|^2 \right) 
\leq \kappa^2\,.
$$
(This means that the threshold for the reassignment rule must be set in accordance with the rate of change of the amplitudes and the instantaneous frequencies of the individual constituent components in the clean signal, as well as with the level of the noise -- both eminently reasonable restrictions.) We shall see below how these restrictions will come into play. 

Let us first prove some technical Lemmas. 

\begin{lemma}\label{Lemma:division}
Fix $J\in\NN$ and $\kappa>0$. Denote
$S_\kappa:=\{\boldsymbol{r}\in S^{J-1}\, ; \,\boldsymbol{r}^\intercal
\boldsymbol{v}>\kappa\}$.
For $\boldsymbol{u}\in \CC^J$ and $\boldsymbol{v}\in\RR^J$, we have
\begin{equation}\label{Lemma:boundProjection}
\frac{1}{|S_\kappa|} \int_{S_\kappa}
\frac{\boldsymbol{r}^\intercal\boldsymbol{u}}
{\boldsymbol{r}^\intercal \boldsymbol{v}} \ud \boldsymbol{r}
=\mathfrak{p}_{\boldsymbol{v}}\boldsymbol{u}
:=\mathfrak{p}_{\boldsymbol{v}}\Re \boldsymbol{u}
+i\mathfrak{p}_{\boldsymbol{v}}\Im \boldsymbol{u},
\end{equation}
where $\Re \boldsymbol{u}$ is the real part of $\boldsymbol{u}$, 
$\Im \boldsymbol{u}$ is the imaginary part of $\boldsymbol{u}$ 
and $\mathfrak{p}_{\boldsymbol{v}}(\boldsymbol{w})$ is the component  
of the vector $\boldsymbol{w}$ along the direction of $\boldsymbol{v}$,
$\mathfrak{p}_{\boldsymbol{v}}(\boldsymbol{w}):= 
\boldsymbol{v}^\intercal\boldsymbol{w}/\|\boldsymbol{v}\|$.
Furthermore
\begin{equation}\label{Lemma:boundSquarePerp}
\frac{1}{|S_\kappa|} \int_{S_\kappa}
\left| \frac{\boldsymbol{r}^\intercal \boldsymbol{u}}
{\boldsymbol{r}^\intercal\boldsymbol{v}}\right|^2 \ud \boldsymbol{r}
=|\mathfrak{p}_{\boldsymbol{v}}\boldsymbol{u}|^2
+c\,\frac{\|\mathcal{P}^\perp_{\boldsymbol{v}}\boldsymbol{u}\|_2^2}{J-1},
\end{equation}
where $\mathcal{P}^\perp_{\boldsymbol{v}}$ 
is the projection operator onto the subspace perpendicular to $\boldsymbol{v}$ and 
$$
c=\frac{2\Gamma((J-1)/2)}{\sqrt{\pi}\Gamma(J/2)}
\int_\kappa^1\frac{(1-x^2)^{(J-1)/2}}{x^2}\ud x
\approx \frac{2\sqrt{2}}{\sqrt{\pi J}\kappa}.
$$
\end{lemma}
\begin{proof}
Without loss of generality, we can assume that 
$\boldsymbol{u}\in\RR^p$ and $\|\boldsymbol{v}\|=1$. 
We can find 
$\mathcal{R}\in SO(J)$ so that 
$\mathcal{R}\boldsymbol{v}=\boldsymbol{e}_1$, where 
$\boldsymbol{e}_1:=[1,0,\ldots,0]^\intercal \in S^{J-1}$. 
Under this change of variable, we write
$$
\mathcal{R}\boldsymbol{u}=[d_1,d_2,\ldots,d_J]^\intercal\in\RR^J,
$$
where
$$
d_1=\boldsymbol{u}^\intercal\boldsymbol{v}=\mathfrak{p}_{\boldsymbol{v}}(\boldsymbol{u}).
$$
As a result, we have
\begin{align*}
&\frac{1}{|S_\kappa|} \int_{S_\kappa}
\frac{\boldsymbol{r}^\intercal \boldsymbol{u}}
{\boldsymbol{r}^\intercal \boldsymbol{v}}
\ud \boldsymbol{r} = 
\frac{1}{|S_\kappa|} \int_{S_\kappa}
\frac{(\mathcal{R}\boldsymbol{r})^\intercal \mathcal{R}\boldsymbol{u}}
{(\mathcal{R}\boldsymbol{r})^\intercal \mathcal{R}\boldsymbol{v}}
\ud \boldsymbol{r} \\
=\,&\frac{1}{|S_\kappa|}\int_{\{\boldsymbol{r}\in S^{J-1};\,r_1>\kappa\}}
\frac{\boldsymbol{r}^\intercal [d_1,d_2,\ldots,d_J]^\intercal}
{\boldsymbol{r}^\intercal\boldsymbol{e}_1}\ud \boldsymbol{r}\\
=\,&\frac{1}{|S_\kappa|}\int_{\{\boldsymbol{r}\in S^{J-1};\,r_1>\kappa\}}
\frac{\sum_{j=1}^J r_j d_j}{r_1}\ud \boldsymbol{r},
\end{align*}
which becomes $d_1=\mathfrak{p}_{\boldsymbol{v}}\boldsymbol{u}$ 
since, for $j \neq 1$, $\frac{r_j}{r_1}$ is an odd function of $r_j$, and the 
domain of integration is invariant under sign reversal of $r_j$. 
For the second part, note that by the same change of variable, we have
\begin{align*}
&\frac{1}{|S_\kappa|}\int_{S_\kappa}
\left|\frac{\boldsymbol{r}^\intercal\boldsymbol{u}}
{\boldsymbol{r}^\intercal\boldsymbol{v}}\right|^2 \ud \boldsymbol{r}\\
=\,&\frac{1}{|S_\kappa|}
\int_{\{\boldsymbol{r}\in S^{J-1};\,r_1>\kappa\}}
\frac{\left(\sum_{j=1}^J r_j d_j\right)^2}{r^2_1}\ud \boldsymbol{r}\\
=\,&\frac{1}{|S_\kappa|}
\int_{\{\boldsymbol{r}\in S^{J-1};\,r_1>\kappa\}}
\frac{\sum_{j=1}^J r^2_j d^2_j+2\sum_{i\neq j}r_i r_j d_i d_j}{r^2_1}
\ud \boldsymbol{r}\\
=\,&d_1^2+\sum_{j=2}^J d^2_j \frac{1}{|S_\kappa|}
\int_{\{\boldsymbol{r}\in S^{J-1};\,r_1>\kappa\}}
\frac{r^2_j}{r^2_1} \ud \boldsymbol{r},
\end{align*}
where the last equality holds because 
$\frac{1}{|S_\kappa|} \int_{\{\boldsymbol{r}\in S^{J-1};\,r_1>\kappa\}}
\frac{\sum_{i\neq j} r_i r_j d_i d_j}{r^2_1}\ud \boldsymbol{r}=0$ 
since in each term $\frac{r_i r_j}{r^2_1}$ 
with $i \neq j$, there is at least one index different from 1, so that this term 
changes sign when it is mirrored with respect to that index, while
the domain of integration is invariant under this mirroring operation. 
For the other terms, note that when $j=2,\ldots,p$, symmetry arguments imply that
\begin{align*}
&\frac{1}{|S_\kappa|}
\int_{\{\boldsymbol{r}\in S^{J-1};\,r_1>\kappa\}}
\frac{r^2_j}{r^2_1}\ud \boldsymbol{r}\\
=\,&\frac{1}{|S_\kappa|}
\frac{1}{J-1} \int_{\{\boldsymbol{r}\in S^{J-1};\,r_1>\kappa\}}\frac{\sum_{l=2}^pr^2_j}{r^2_1}\ud \boldsymbol{r}\\
=\,&\frac{1}{|S_\kappa|}\frac{1}{J-1}\int_{\{\boldsymbol{r}\in S^{J-1};\,r_1>\kappa\}}\frac{1-r^2_1}{r^2_1}\ud \boldsymbol{r}.
\end{align*}
To evaluate the last term, we use spherical coordinates in $J$ dimensions. 
Rewrite $\boldsymbol{r}=[r_1,\ldots,r_J]^\intercal \in S^{J-1}$ as
\begin{align*}\left\{
\begin{array}{l}
r_1=\cos\varphi_1\\
r_2=\sin\varphi_1\cos\varphi_2\\
\vdots\\
r_{J-1}=\sin\varphi_1\ldots\sin\varphi_{J-2}\cos\varphi_{J-1}\\
r_{J}=\sin\varphi_1\ldots\sin\varphi_{J-2}\sin\varphi_{J-1},
\end{array}\right.
\end{align*}
where $\varphi_1,\ldots,\varphi_{J-2}\in [0,\pi)$ and $\varphi_{J-1}\in [0,2\pi)$. 
In this coordinate system, the volume form becomes 
$\ud \boldsymbol{r}=(\sin^{J-2}\varphi_1) (\sin^{J-3}\varphi_2)\ldots
(\sin\varphi_{J-2})\,\ud \varphi_{J-1} \,\ud \varphi_{J-2}\ldots \ud\varphi_1$; 
 we obtain thus
\begin{align*}
&\int_{\{\boldsymbol{r}\in S^{J-1};\,r_1>\kappa\}}
\frac{1}{r^2_1}\,\ud \boldsymbol{r}\\
=\,&2\int_{I_\kappa} 
\int_0^\pi \ldots \int_0^\pi \int_{0}^{2\pi} 
\frac{\sin^{J-2}\varphi_1 \, \sin^{J-3}\varphi_2\,\ldots,\sin\varphi_{J-2}}
{\cos^2 \varphi_1}\,\ud \varphi_{J-1}\,\ud \varphi_{J-2} \ldots \,\ud\varphi_1\\
=\,&2|S^{J-2}|\int_{I_\kappa}  
\frac{\sin^{J-2}\varphi_1 }{\cos^2 \varphi_1} \,\ud\varphi_1,
\end{align*}
where $I_\kappa=\{\varphi_1\in[0,\pi/2];\, \cos\varphi_1>\kappa\}$. Similarly, we have
\[
\int_{\{\boldsymbol{r}\in S^{J-1};\,r_1>\kappa\}}\ud \boldsymbol{r}
=2\,|S^{J-2}|\int_{I_\kappa}  \sin^{J-2}\varphi_1 \,\ud\varphi_1.
\]
By putting the above together, we have finished the claim since
\begin{align*}
\frac{1}{|S_\kappa|}
\int_{S_\kappa} 
\left|\frac{\boldsymbol{r}^\intercal\boldsymbol{u}}
{\boldsymbol{r}^\intercal \boldsymbol{v}}\right|^2 \ud\boldsymbol{r}
=\,d_1^2
+c\,\frac{\sum_{j=2}^J d^2_j}{J-1}
=|\mathfrak{p}_{\boldsymbol{v}}\boldsymbol{u}|^2
+c \, \frac{\|\mathcal{P}^\perp_{\boldsymbol{v}}\boldsymbol{u}\|_2^2}{J-1},
\end{align*}
where
\begin{equation}\label{proof:lemma:square:constant}
c=\frac{2|S^{J-2}|}{|S_\kappa|}
\int_{I_\kappa}  \sin^{J-2}\varphi_1 
\left(\frac{1}{\cos^2 \varphi_1}-1\right) \ud\varphi_1
=\frac{2\Gamma((J-1)/2)}{\sqrt{\pi}\Gamma(J/2)}\int_{I_\kappa} \frac{ \sin^{J}
\varphi }{\cos^2 \varphi}\,
\ud\varphi.
\end{equation}
Notice that the Gamma function ratio 
$\frac{\Gamma((J-1)/2)}{\Gamma(J/2)}$ can be asymptotically 
approximated by $(J/2)^{-1/2}$ as $J\to \infty$ and that
\[
\int_{I_\kappa} \frac{ \sin^{J}\varphi }{\cos^2 \varphi}\ud\varphi
=\int_\kappa^1\frac{(1-u^2)^{(J-1)/2}}{u^2}\ud u
=\int_\kappa^1\frac{1}{u^2}(1+O(u^2))\ud u
=\frac{1}{\kappa}+O(1).
\]
It follows that $c$ is approximately
\begin{equation}\label{proof:Lemma:approximationOfc}
c\approx \frac{2\sqrt{2}}{\sqrt{\pi J}\kappa}.
\end{equation}
\end{proof}
We are now ready to study the statistical behavior of $\omega_Y^{(\psi^{[\boldsymbol{r}]})}(a,b)$  as the unit vector $\boldsymbol{r}$ is picked randomly, uniformly in $\widetilde{S}_\kappa^{(\ell)} =\{\boldsymbol{r}\in S^{J-1}\,;\, \boldsymbol{r}^\intercal\left(e^{-i2\pi\varphi_\ell(b)}\boldsymbol{Q}_\ell(a,b)+\boldsymbol{\epsilon(a,b)}+\sigma\boldsymbol{\Phi(a,b)}\right)>2\kappa\}\subset S^{J-1}$. In the next proposition, we keep $\ell$, $b$ and $a$ fixed, on the understanding that $a\in Z_\ell(b)$. To ease up on notation, we shall suppress $(a,b)$ and $\ell$ in the notation, and use $\omega_Y^{(\psi^{[\boldsymbol{r}]})}$, $\boldsymbol{Q}$, $\varphi(b)$, $\widetilde{\boldsymbol{\epsilon}}$, $\widetilde{\boldsymbol{\Phi}}$, $\widetilde{S}_\kappa$, etc, to denote $\omega_Y^{(\psi^{\boldsymbol{r}})}(a,b)$, $\boldsymbol{Q}_\ell(a,b)$, $\varphi_\ell(b)$, $\widetilde{\boldsymbol{\epsilon}}(a,b)$, $\widetilde{\boldsymbol{\Phi}}(a,b)$, $\widetilde{S}_\kappa^{(\ell)}$, etc.
\begin{prop}
Fix a realization of $\Phi$, $\kappa>0$, $b\in\RR$ and $a\in Z_\ell(b)$. Assume that $\boldsymbol{r}$ is sampled uniformly from $\widetilde{S}_\kappa=\{\boldsymbol{r}\in S^{J-1}\,;\, \boldsymbol{r}^\intercal\left(e^{-i2\pi\varphi(b)}\boldsymbol{Q}+\boldsymbol{\epsilon}+\sigma\boldsymbol{\Phi}\right)>2\kappa\}\subset S^{J-1}$. When $\|\boldsymbol{\epsilon}+\sigma\boldsymbol{\Phi}\|_2^2<\kappa$, we have
\begin{align}
\mathbb{E}_{\boldsymbol{r}}\omega_Y^{(\psi^{[\boldsymbol{r}]})}
&=\varphi'(b)+e^{-i 2\pi \varphi(b)} \mathfrak{p}_{\boldsymbol{Q}}
\left(\widetilde{\boldsymbol{\epsilon}}+\sigma\widetilde{\boldsymbol{\Phi}}
-\varphi'(b)[\boldsymbol{\epsilon}+\sigma\boldsymbol{\Phi}]\right)+E_1,\label{Lemma:ExpectationOverrOfOmega}
\end{align}
where $\mathbb{E}_{\boldsymbol{r}}$ is the expectation of $\omega_Y^{(\psi^{[\boldsymbol{r}]})}(a,b)$ as $\boldsymbol{r}$ is sampled randomly and uniformly from $\widetilde{S}_\kappa$, 
and $E_1$ is bounded by
\begin{equation}\label{proof:Lemma2:Expectation:ErrorBound}
|E_1|\leq
\frac{1}{2}\,
\left(\left[ 1-\frac{c}{J-1} \right] 
|\mathfrak{p}_{\boldsymbol{Q}}
\left(\widetilde{\boldsymbol{\epsilon}}+\sigma\widetilde{\boldsymbol{\Phi}}-
\varphi'(b)[\boldsymbol{\epsilon}+\sigma\boldsymbol{\Phi}]\right)|^2
+ c\, \frac{\|
\widetilde{\boldsymbol{\epsilon}}+\sigma\widetilde{\boldsymbol{\Phi}}-
\varphi'(b)[\boldsymbol{\epsilon}+\sigma\boldsymbol{\Phi}]\|^2}{J-1}\right)^{1/2}
~. 
\end{equation}
Furthermore we have
\begin{align*}
\text{Var}_{\boldsymbol{r}}\,\omega_Y^{(\psi^{[\boldsymbol{r}]})}
\leq&\,\frac{5}{2}\,\left(
\left[ 1-\frac{c}{J-1} \right] 
|\mathfrak{p}_{\boldsymbol{Q}}
\left(\widetilde{\boldsymbol{\epsilon}}+\sigma\widetilde{\boldsymbol{\Phi}}-
\varphi'(b)[\boldsymbol{\epsilon}+\sigma\boldsymbol{\Phi}]\right)|^2
+ c\, \frac{\|
\widetilde{\boldsymbol{\epsilon}}+\sigma\widetilde{\boldsymbol{\Phi}}-
\varphi'(b)[\boldsymbol{\epsilon}+\sigma\boldsymbol{\Phi}]\|^2}{J-1}\right)~. 
\end{align*}
where $\text{Var}_{\boldsymbol{r}}$ is the variance of 
$\omega_Y^{(\psi^{[\boldsymbol{r}]})}$ over $\widetilde{S}_\kappa$.
\end{prop}
Before the proof, we have the following remark about the Proposition.
\begin{remark}
In the statement of this proposition, we encounter several times the
expression $\left|\mathfrak{p}_{\boldsymbol{Q}}\boldsymbol{V} \right|$
(using the shorthand notation $\boldsymbol{V}
=\widetilde{\boldsymbol{\epsilon}}
+\sigma\widetilde{\boldsymbol{\Phi}}
-\varphi'(b)[\boldsymbol{\epsilon}+\sigma\boldsymbol{\Phi}]\,$),
which can be bounded by 
$\| \boldsymbol{V} \|$. 
In practice, however, the term $\left|\mathfrak{p}_{\boldsymbol{Q}}\boldsymbol{V} \right|$ will likely be significantly smaller
than its norm, with high probability if $J$ is large. Indeed,
the vector $\boldsymbol{Q}$ is fixed,
while the vector 
$\boldsymbol{V}$ is a random vector
in $J$ dimensions, depending on the random realization of the noise function 
$\Phi$, which is much more likely than not to lie in a region near the equator, 
perpendicular to $\boldsymbol{Q}$, 
since this region contributes the lion share of the sphere ``area'' (really
a $J-1$-dimensional volume),
increasingly so as $J$ increases. 
Denoting $I_\gamma:=\{\varphi_1\in[0,\pi/2]\,;\, 0\leq\cos(\varphi_1)<\gamma\}$,
we have indeed 
\begin{align*}
&|\{\boldsymbol{r}\in S^{J-1};\,0\leq r_1<\gamma\}|\\
=\,&\int_{I_\gamma}\int_0^\pi\ldots\int_0^\pi\int_{0}^{2\pi} 
(\sin^{J-2}\varphi_1) (\sin^{J-3}\varphi_2)\,\ldots(\sin\varphi_{J-2})\,
\ud \varphi_{J-1}\,\ud\varphi_{J-2}\,\ldots \,\ud\varphi_1\\
=\,&|S^{J-2}|\int_{I_\gamma}\sin^{J-2}\varphi_1 \ud\varphi_1.
\end{align*}
Consequently,
\[
\frac{|\{\boldsymbol{r}\in S^{J-1};\,r_1<\gamma\}|}
{|\{\boldsymbol{r}\in S^{J-1};\,0\leq r_1 \leq 1\}|}
=\frac{\int_{I_\gamma}\sin^{J-2}\varphi\,\ud\varphi}
{\int_{0}^{\pi/2}\sin^{J-2}\varphi\,\ud\varphi}
=1-\frac{\int_{\gamma}^{1}(1-u^2)^{(J-3)/2}\ud u }
{\int_{0}^1(1-u^2)^{(J-3)/2}\ud u},
\]
which approaches $1$ as $J$ increases to $\infty$.
\end{remark}

\begin{proof} (of the Proposition.)
To simplify the notation in the computation, we set
$\boldsymbol{A}:= \varphi'(b)e^{i2\pi\varphi(b)}\boldsymbol{Q} $, 
$\boldsymbol{a}:= \widetilde{\boldsymbol{\epsilon}}
+\sigma\widetilde{\boldsymbol{\Phi}} $,
$\boldsymbol{B}:= e^{i2\pi\varphi(b)}\boldsymbol{Q} = \boldsymbol{A}/ \varphi'(b)$,
$\boldsymbol{b}:= \boldsymbol{\epsilon}+\sigma\boldsymbol{\Phi} $.

By the assumption that $\boldsymbol{r}$ is sampled uniformly from $S_\kappa$, 
we have
\begin{align}
\mathbb{E}_{\boldsymbol{r}}\omega_Y^{(\psi^{[\boldsymbol{r}]})}=
\,&\frac{1}{|S_\kappa|}\int_{S_\kappa}
\frac{\boldsymbol{r}^\intercal\left(\boldsymbol{A}+\boldsymbol{a}\right)}
{\boldsymbol{r}^\intercal\left(\boldsymbol{B}+\boldsymbol{b}\right)}
\ud \boldsymbol{r}=
\,\frac{1}{|S_\kappa|}\int_{S_\kappa}
\frac{\boldsymbol{r}^\intercal\left(\varphi'(b)\boldsymbol{B}+\boldsymbol{a}\right)}
{\boldsymbol{r}^\intercal\left(\boldsymbol{B}+\boldsymbol{b}\right)}
\ud \boldsymbol{r}
\nonumber\\
=\,&\frac{\varphi'(b)}{|S_\kappa|}
\int_{S_\kappa}\left[\,1\,+\,\frac{\boldsymbol{r}^\intercal 
\left(\frac{\boldsymbol{a}}{\varphi'(b)}-\boldsymbol{b} \right)}
{\boldsymbol{r}^\intercal\left(\boldsymbol{B}+\boldsymbol{b}\right)}\,\right]\,
\ud \boldsymbol{r}\nonumber\\
=\,&\varphi'(b)+\frac{1}{|S_\kappa|}
\int_{S_\kappa}
\frac{\boldsymbol{r}^\intercal
\left(\boldsymbol{a}-\varphi'(b)\boldsymbol{b}\right)}
{\boldsymbol{r}^\intercal\left(\boldsymbol{B}+\boldsymbol{b}\right)}
\,\ud \boldsymbol{r}\,.\label{proof:Lemma:Integrant:Expectation}
\end{align}
We next use the identity
\[
\frac{\boldsymbol{r}^\intercal
\left(\boldsymbol{a}-\varphi'(b)\boldsymbol{b}\right)}
{\boldsymbol{r}^\intercal\left(\boldsymbol{B}+\boldsymbol{b}\right)}
= \frac{\boldsymbol{r}^\intercal
\left(\boldsymbol{a}-\varphi'(b)\boldsymbol{b}\right)}
{\boldsymbol{r}^\intercal\boldsymbol{B}}\,-\,
\frac{\boldsymbol{r}^\intercal \boldsymbol{b}}
{\boldsymbol{r}^\intercal\boldsymbol{B}}\,
\frac{\boldsymbol{r}^\intercal\left(\boldsymbol{a}-\varphi'(b)\boldsymbol{b}\right)}
{\boldsymbol{r}^\intercal\left(\boldsymbol{B}+\boldsymbol{b}\right)},
\]
combining it with Lemma \ref{Lemma:division} (since $\boldsymbol{Q}\in\RR^p$),  
to obtain
\begin{align}
\mathbb{E}_{\boldsymbol{r}}\omega_Y^{(\psi^{[\boldsymbol{r}]})}
=\,&\varphi_k'(b)+
e^{-i2\pi\varphi(b)}\mathfrak{p}_{\boldsymbol{Q}}\left(\widetilde{\boldsymbol{\epsilon}}+\sigma\widetilde{\boldsymbol{\Phi}}-
\varphi'(b)[\boldsymbol{\epsilon}+\sigma\boldsymbol{\Phi}]\right)+E_1,\label{proof:Lemma2:Expectation:ErrorExpression0}
\end{align}
where
\begin{align}
E_1:=-\,\frac{1}{|S_\kappa|}\int_{S_\kappa}
\frac{\boldsymbol{r}^\intercal \boldsymbol{b}}
{\boldsymbol{r}^\intercal\boldsymbol{B}}\,
\frac{\boldsymbol{r}^\intercal\left(\boldsymbol{a}-\varphi'(b)\boldsymbol{b}\right)}
{\boldsymbol{r}^\intercal\left(\boldsymbol{B}+\boldsymbol{b}\right)}
\, \ud\boldsymbol{r}.\label{proof:Lemma2:Expectation:ErrorExpression}
\end{align}
Note that by the assumptions that 
$\|\boldsymbol{\epsilon}+\sigma\boldsymbol{\Phi}\|_2\leq \kappa$ 
and 
$|\boldsymbol{r}^\intercal\left(e^{i2\pi\varphi(b)}\boldsymbol{Q}+\boldsymbol{\epsilon}+\sigma\boldsymbol{\Phi}\right)|>2\kappa$, we 
have
\[
\frac{|\boldsymbol{r}^\intercal \boldsymbol{b}|}
{|\boldsymbol{r}^\intercal\left(\boldsymbol{B}+\boldsymbol{b}\right)|} 
< \frac{1}{2}\,,
\]
so that
\begin{align}
|E_1|\leq \frac{1}{|S_\kappa|}\int_{S_\kappa}
\frac{|\boldsymbol{r}^\intercal\left(\boldsymbol{a}-\varphi'(b)\boldsymbol{b}\right)|}
{2|\boldsymbol{r}^\intercal\boldsymbol{Q}|} \ud\boldsymbol{r}.\nonumber
\end{align}
Next, we apply the Cauchy-Schwarz inequality to this integral, together
with Lemma \ref{Lemma:division}, which leads to
\begin{align*}
|E_1|&\leq \frac{1}{|S_\kappa|}
\left[ \int_{S_\kappa } \ud\boldsymbol{r}\right]^{1/2}
\left[ \int_{S_\kappa } \frac{|\boldsymbol{r}^\intercal\left(\boldsymbol{a}-\varphi'(b)\boldsymbol{b}\right)|^2}
{2|\boldsymbol{r}^\intercal\boldsymbol{Q}|^2}\right]^{1/2}\\
&\leq \frac{1}{2}\,
\left(  
|\mathfrak{p}_{\boldsymbol{Q}}
\left(\boldsymbol{a}-\varphi'(b)\boldsymbol{b}\right)|^2
+ c\, \frac{\|\mathcal{P}_{\boldsymbol{Q}}^\perp
\left(\boldsymbol{a}-\varphi'(b)\boldsymbol{b}\right)\|^2}{J-1}
\right)^{1/2}\\
&= \frac{1}{2}\,
\left( \left[ 1-\frac{c}{J-1} \right] |\mathfrak{p}_{\boldsymbol{Q}}
\left(\boldsymbol{a}-\varphi'(b)\boldsymbol{b}\right)|^2 +
c\,\frac{\|\boldsymbol{a}-\varphi'(b)\boldsymbol{b}\|^2}{J-1} \right)^{1/2}\\
&= \frac{1}{2}\,
\left(\left[ 1-\frac{c}{J-1} \right] 
|\mathfrak{p}_{\boldsymbol{Q}}
\left(\widetilde{\boldsymbol{\epsilon}}+\sigma\widetilde{\boldsymbol{\Phi}}-
\varphi'(b)[\boldsymbol{\epsilon}+\sigma\boldsymbol{\Phi}]\right)|^2
+ c\, \frac{\|
\widetilde{\boldsymbol{\epsilon}}+\sigma\widetilde{\boldsymbol{\Phi}}-
\varphi'(b)[\boldsymbol{\epsilon}+\sigma\boldsymbol{\Phi}]\|^2}{J-1}\right)^{1/2}~. 
\end{align*}

The variance can be evaluated in the same manner. 
Noting that for a random variable $X$, 
$\text{Var}X=\text{Var}(X-c)$ for any constant $c$, we have
\begin{align}
\text{Var}_{\boldsymbol{r}}\omega_Y^{(\psi^{[\boldsymbol{r}]})}=
&\,\text{Var}_{\boldsymbol{r}}[\omega_Y^{(\psi^{[\boldsymbol{r}]})}-\varphi'(b)]\nonumber\\
=&\,\frac{1}{|S_\kappa|}\int_{S_\kappa}
\left|\frac
{\boldsymbol{r}^\intercal\left(\boldsymbol{a}-\varphi'(b)\boldsymbol{b}\right)}
{\boldsymbol{r}^\intercal\left(\boldsymbol{B}+\boldsymbol{b}\right)}\right|^2\ud \boldsymbol{r}
-\left|\frac{1}{|S_\kappa|}\int_{S_\kappa}
\frac{\boldsymbol{r}^\intercal\left(\boldsymbol{a}-\varphi'(b)\boldsymbol{b}\right)}{\boldsymbol{r}^\intercal\left(\boldsymbol{B}+\boldsymbol{b}\right)}
\,\ud \boldsymbol{r}\right|^2.\nonumber\\
\leq& \,\frac{1}{|S_\kappa|}\int_{S_\kappa}
\left|\frac
{\boldsymbol{r}^\intercal\left(\boldsymbol{a}-\varphi'(b)\boldsymbol{b}\right)}
{\boldsymbol{r}^\intercal\left(\boldsymbol{B}+\boldsymbol{b}\right)}\right|^2\ud \boldsymbol{r}\label{Lemma:EvaluateVariance1}
\end{align}
This last expression (\ref{Lemma:EvaluateVariance1}) can be bounded by
\begin{align*}
&\frac{2}{|S_\kappa|}
\int_{S_\kappa}\left(\left|\frac
{\boldsymbol{r}^\intercal\left(\boldsymbol{a}-\varphi'(b)\boldsymbol{b}\right)}
{ \boldsymbol{r}^\intercal\boldsymbol{B}}\right|^2 + 
\left| \frac
{\boldsymbol{r}^\intercal\boldsymbol{b}}
{ \boldsymbol{r}^\intercal\boldsymbol{B}}\,
\frac
{\boldsymbol{r}^\intercal\left(\boldsymbol{a}-\varphi'(b)\boldsymbol{b}\right)}
{\boldsymbol{r}^\intercal\left(\boldsymbol{B}+\boldsymbol{b}\right)}
\right|^2 \right)\, \ud\boldsymbol{r}\\
&\quad \leq \frac{2}{|S_\kappa|}\,\frac{5}{4}
\int_{S_\kappa}\left|\frac
{\boldsymbol{r}^\intercal\left(\boldsymbol{a}-\varphi'(b)\boldsymbol{b}\right)}
{ \boldsymbol{r}^\intercal\boldsymbol{B}}\right|^2\, \ud\boldsymbol{r}~.
\end{align*}
We have encountered this exact same integral before, and bounded it by invoking
Lemma \ref{Lemma:division}. We thus obtain
\begin{align*}
\text{Var}_{\boldsymbol{r}}\omega_Y^{(\psi^{[\boldsymbol{r}]})}
&\leq \frac{5}{2} \,\left(
 \left[ 1-\frac{c}{J-1} \right] |\mathfrak{p}_{\boldsymbol{Q}}
\left(\boldsymbol{a}-\varphi'(b)\boldsymbol{b}\right)|^2 +
c\,\frac{\|\boldsymbol{a}-\varphi'(b)\boldsymbol{b}\|^2}{J-1} \right)\\
&= \frac{5}{2}\,\left(
\left[ 1-\frac{c}{J-1} \right] 
|\mathfrak{p}_{\boldsymbol{Q}}
\left(\widetilde{\boldsymbol{\epsilon}}+\sigma\widetilde{\boldsymbol{\Phi}}-
\varphi'(b)[\boldsymbol{\epsilon}+\sigma\boldsymbol{\Phi}]\right)|^2
+ c\, \frac{\|
\widetilde{\boldsymbol{\epsilon}}+\sigma\widetilde{\boldsymbol{\Phi}}-
\varphi'(b)[\boldsymbol{\epsilon}+\sigma\boldsymbol{\Phi}]\|^2}{J-1}\right)~. 
\end{align*}
\end{proof}

This concludes this section concerning the
details for the technical estimates
in section 3 of the main paper. 

\section*{ESM-4. Numerical results}\label{ESM-3}

As described in the main paper, we consider both CWT and STFT-based ConceFT representations. 
In both cases, the orthogonal family of reference functions
(wavelets for the CWT, windows for the STFT) are the eigenfunctions, up to 
a certain order, of a time-frequency localization operator that is particularly
well suited to the CWT or STFT framework \cite{DaubPaul:1988,Olhede_Walden:2002,
Daubechies:1988,Xiao_Flandrin:2007}. 
Figure \ref{fig:ESM-4.1} below shows the shape and
size of TF domains of this type. In both cases, the shapes correspond to
a two-parameter family, and the localization operators behave approximately like
projection operators.
More precisely, once the parameters $\Lambda$ determining the shape are picked, there
is a natural family of (commuting) operators $T^{(\Lambda,R)}$ and an orthonormal family
of functions $\psi_j^{(\Lambda)}$ such that
\[
T^{(\Lambda,R)} \psi_j^{(\Lambda)} = E_j^{(\Lambda,R)}\,\psi_j^{(\Lambda)},
\]
where the eigenvalues $E_j^{(\Lambda,R)}$, all between 0 and 1, constitute a
strictly decreasing sequence, tending to $0$ as $j$ tends to $\infty$; for fixed
$\Lambda$ and $j$, each $E_j^{(\Lambda,R)}$ increases with $R$, tending to 1 as 
$R$ (which indicates the size of the region characterized by $\Lambda$)
tends to $\infty$. The eigenfunctions themselves (which do {\em not} depend on $R$)
are scaled and possibly
chirped Hermite functions for the STFT case, and Morse functions in the CWT case. 

\begin{figure}[h!]
\centering
\includegraphics[width=\textwidth]{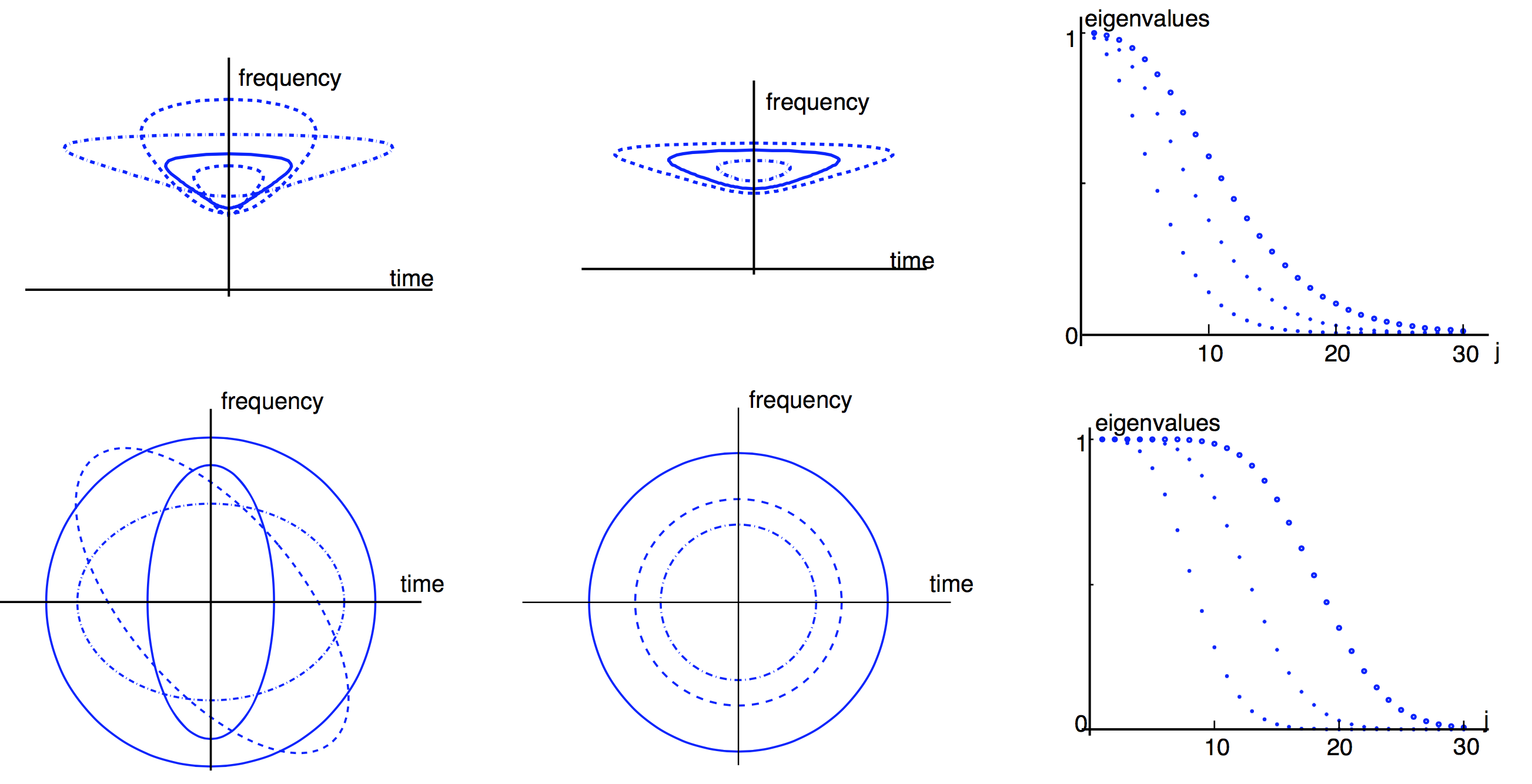}
\caption{Localization domains in the TF plane for the  reference windows or CWT reference wavelets:
Top: CWT, bottom: STFT. From left to right: different shapes of the TF domain, corresponding to
different parameter choices $\Lambda$; 
different sizes of one domain shape, corresponding (for one fixed $\Lambda$) to different $R$; eigenvalues $E_j^{(\Lambda,R)}$, for different $R$.\label{fig:ESM-4.1}}
\end{figure}

It seems natural to pick these special orthonormal families, since each family provides,
in some sense (made precise in \cite{DaubPaul:1988,Olhede_Walden:2002,
Daubechies:1988,Xiao_Flandrin:2007}) the ``best'' localization, simultaneously,
by different orthonormal functions, for one shared time-frequency domain. (A similar
reason underlies the choice, in standard multi-taper methods for spectral estimation,
of the prolate spheroidal wave functions for the taper functions \cite{Thomson:1982, Percival:1993, Babadi_Brown:2014}.) 
However, the method does not depend on 
these particular choices, and it is not only conceivable, but indeed likely, that for
particular applications, other choices may be more suitable and give better results. 

\subsection*{ESM-4a Data simulation\label{sec:data}}

Figure \ref{fig:ESM-4.2}
below shows the graph of (the restriction to $[15,40]$ of)
another signal $s^\ast \in \mathcal{C}$. 
This signal is used in the main paper to illustrate 
the action of ConceFT on a signal from $\mathcal{C}$ that
has played no role in calibrating the ConceFT parameters (unlike $s$).

\begin{figure}[h!]
\begin{centering}
\includegraphics[width=.95\textwidth]{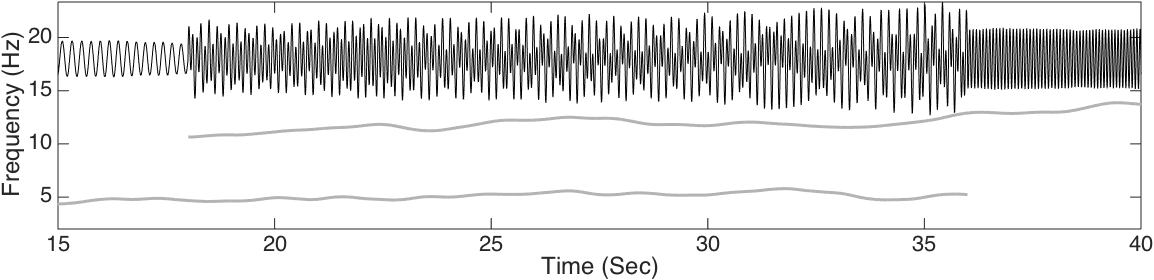}
\end{centering}
\caption{\label{fig:ESM-4.2} Another signal $s^\ast$ (in black) in $\mathcal{C}$, and the corresponding instantaneous frequencies (in gray) of the two components, restricted to the time interval $[15,40]$.} 
\end{figure}

Figure \ref{fig:ESM-4.3} plots a realization of 
$Y^\ast(t)=s^\ast(t)+\sigma\xi(t)$ for each
of the three noise processes (Gaussian, ARMA(1,1) and Poisson), 
restricted to the subinterval $[15,40]$.
\begin{figure}[h!]
\begin{centering}
\includegraphics[width=.95\textwidth]{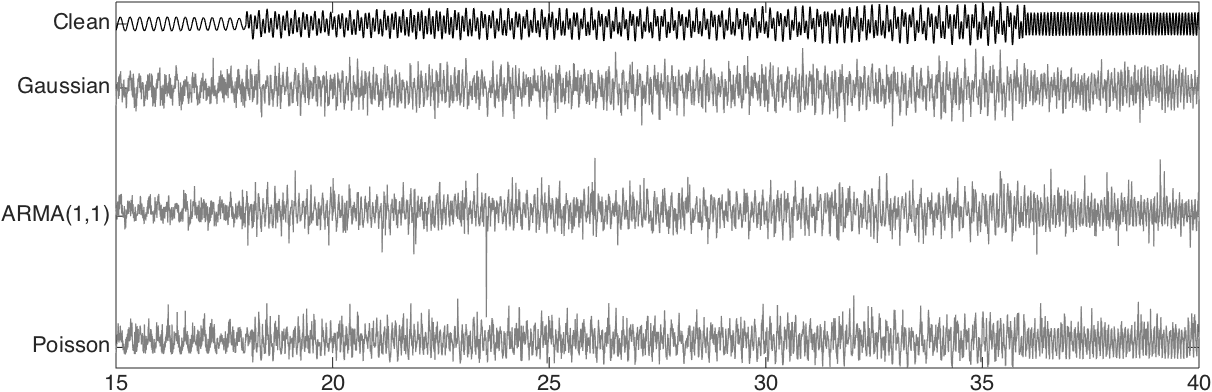}
\end{centering}
\caption{\label{fig:ESM-4.3} The restrictions to $[15,40]$ of the noisy signal $Y^\ast=s^\ast+\sigma \xi$ (2nd row to bottom), where $s^{\ast}$ is the clean signal from the previous figure (plotted again in the top row) and where the added noise is Gaussian, ARMA, or Poisson noise (in order, from); in each case $\sigma$ is picked so that the noisy signal has 0 dB SNR (signal to noise ratio). The four plots are at the same vertical scale.}
\end{figure}

\subsection*{ESM-4b Performance evaluation}

To assess the performance of ConceFT, we must compare the time-varying Power Spectrum (tvPS) $\widetilde{\texttt{P}}_Y$, as estimated via ConceFT, with the {\it ideal time-varying power spectrum} (itvPS) of the clean simulated signal $s$, defined (in a natural interpretation of its construction procedure) as
\[
\texttt{P}_{s}(t,\omega):= \sum_{k=1}^2 A_k^2(t) \delta_{\varphi'_k(t)(\omega)}. \nonumber
\]
Viewing both the itvPS and the tvPS as distributions on the TF-plane, we want to assess, in particular, whether the regions in the TF plane where they each concentrate, coincide or lie close to each other. The {\em Optimal Transport} (OT) distance (also called the Earth Mover distance) is a distance that is designed to do this: given two probability measures on the same set, their OT-distance gives the amount of ``work'' needed to ``deform'' one into the other. A bit more precisely, it computes the total (integral/sum of the product) {\em mass} $\times$ {\em distance traveled} for the transformation (i.e. the transportation plan that minimizes this quantity) that maps one to the other. Because the principle of ConceFT is to ``reassign'' content in the TF plane, keeping the time-variable fixed (see Section 2), we also compute the OT-distance for each individual $t$ (keeping $t$ fixed), and then take the average over all $t \in [0,T]$. This has a fortuitous advantage, in that it reduces the OT-distance computations to 1-dimensional problems, for which there exists a computational short-cut: the standard definition for the OT-distance between probability distributions $\mu$ and $\nu$ on a metric space $({\tt S},d)$ involves an optimization over $\mathcal{P}(\mu,\nu)$, the set of all probability measures on ${\tt S}\times{\tt S}$ that have $\mu$ and $\nu$ as marginals,
\[
d_{\mbox{\footnotesize{OT}}}(\mu,\nu):= \inf_{\rho \in \mathcal{P}(\mu,\nu)}\int\,d(x,y)\,
\ud \rho(x,y)~,
\]
which can be computationally quite expensive. In the one-dimensional case (i.e. when ${\tt S}\subset \RR$, and $d$ is the canonical Euclidean distance, $d(x,y)=|x-y|$), however, it turns out (see e.g. section 2.2 in \cite{Villani:book}) that, defining  $f_\mu(x)=\int_{-\infty}^x\, \ud \mu $ (analogously for $f_{\nu}$), we have
\[
d_{\mbox{\footnotesize{OT}}}(\mu,\nu)= \int_{\mbox{\tt S}}\, |f_\mu(x)-f_\nu(x)|\, \ud x ~.
\]

The OT-distance is defined for {\em probability} distributions, and it is by no means guaranteed that the positive functions $\widetilde{\texttt{P}}_{Y}(t,\cdot)$ and $\texttt{P}_s(t,\cdot)$ have integral 1 for all $t$; for this reason, we normalize them before computing their OT-distance. We may also want to capture (and penalize in the distance metric) possible differences in the total weights of $\widetilde{\texttt{P}}_{Y}(t,\cdot)$ and $\texttt{P}_s(t,\cdot)$; we can introduce a term for this as well. More precisely, assuming that the frequency domain over which $\widetilde{\texttt{P}}_{Y}$ and $\texttt{P}_s$ range is $[0,\Omega]$, and assuming also that $\int_0^T \int_0^{\Omega} \,\widetilde{\texttt{P}}_{Y}(t,\omega)\,\ud \omega \,\ud t= \int_0^T \int_0^{\Omega} \,\texttt{P}_s(t,\omega)\,\ud \omega \,\ud t $ (which can be achieved by multiplying $ \widetilde{\texttt{P}}_{Y}$ with a constant, if necessary), we define
\begin{align*}
\widetilde{p}_{Y}(t,\omega)=\int_{0}^\omega \widetilde{\texttt{P}}_{Y}(t,\xi)\,\ud \xi \quad~&~\quad
p_{s}(t,\omega)= \int_{0}^\omega \texttt{P}_s(t,\xi)\,\ud \xi \\
\widetilde{\rho}_{Y}(t,\omega)=\widetilde{p}_{Y}(t,\omega)/\widetilde{p}_{Y}(t,\Omega)
\quad ~&~ \quad
\rho_{s}(t,\omega)=p_s(t,\omega)/p_s(t,\Omega) \\
\mbox{\tt{D}}_{\alpha}(\widetilde{\texttt{P}}_{Y}, \texttt{P}_s)= \alpha \, \frac{1}{T}\,\int_0^T\, 
\frac{|\widetilde{p}_{Y}(t,\Omega)- p_s(t,\Omega)|}
{\widetilde{p}_{Y}(t,\Omega)+ p_s(t,\Omega) }\,\ud t \,&+\,
(1-\alpha)\,\frac{1}{T}\,\int_0^T\int_0^{\Omega}\,\left|\widetilde{\rho}_{Y}(t,\omega) 
-\rho_{s}(t,\omega) \right|\,\ud \omega \, \ud t
\end{align*}
In practice, we picked $\alpha =0$ in our evaluations, since the corresponding pure OT distance already gave us a reasonable way to quantify how well a tvPS reflected ``its'' itvPS, consistent with our (subjective) appraisals. In concrete computations, the integrals are approximated by sums of the corresponding discretized quantities. 

\subsection*{ESM-4c Parameter Selection}
In this subsection, we report the details of our exploration of the parameter space, leading to our choice of $\beta =30$, $\gamma=9$ and $J=2$ as the optimal one for the CWT-based ConceFT algorithm, when applied to the signal class $\mathcal{C}$.

We applied ConceFT to the noisy signals, with $\gamma=3,4,\cdots,10$ (8 choices); $\beta=20,30,\cdots,70$ (6 choices) and $J=1,2,3,4$ (4 choices). All 192 possible combinations of these options are investigated. For each example and each parameter setting, we applied the ConceFT algorithm 10 times, each time with 10 random projections; the average of the OT distances over these 10 attempts was then computed.

Figure \ref{fig:ESM-4.5} visualizes the results by means of a ``heat map''. In this figure, the $x$-axis is the selection of $\gamma$ and $\beta$, the $y$-axis is $J$ and the color at each entry represents the averaged OT distance for the corresponding choice of the parameters $\gamma$, $\beta$ and $J$; the lighter the color, the smaller the OT distance and hence the better the performance. The Figure shows the averaged OT distance of the ConceFT result for all choices of parameters, for one signal $s^\#$ in $\mathcal{C}$, and three types of noise, giving three heat maps in total. The $x$-coordinate in each heat map cycles through the 6 values of $\beta$ before it moves on to a new value of $\gamma$; Table \ref{table:ESM-4.4} below gives the value of $x$ for each pair of Morse parameters considered. 

\begin{table}[h!]
\begin{center}
\begin{tabular}{|c|c|c|c|c|c|c|c|c|}
\hline
&$\gamma=3$ &$\gamma=4$&$\gamma=5$&$\gamma=6$&$\gamma=7$&$\gamma=8$&$\gamma=9$&$\gamma=10$ \\\hline
$\beta=20$&  1 & 7 &13&19&25&31&37 &43\\\hline
$\beta=30$&  2 &8 &14& 20&26&32&38&44  \\\hline
$\beta=40$&  3 &9 &15& 21&27&33&38&45  \\\hline
$\beta=50$&  4 &10 &16& 22&28&34&39&46  \\\hline
$\beta=60$&  5 &11 &17& 23&29&35&40&47  \\\hline
$\beta=70$&  6 &12&18& 24&30&36&41&48\\\hline
\end{tabular}
\vspace{0.1in}
\caption{The numbers on the $x$ axis and their corresponding Morse parameters.\label{table:ESM-4.4}}
\end{center}
\end{table}
\begin{figure}[h!]
\begin{centering}
\includegraphics[width=\textwidth]{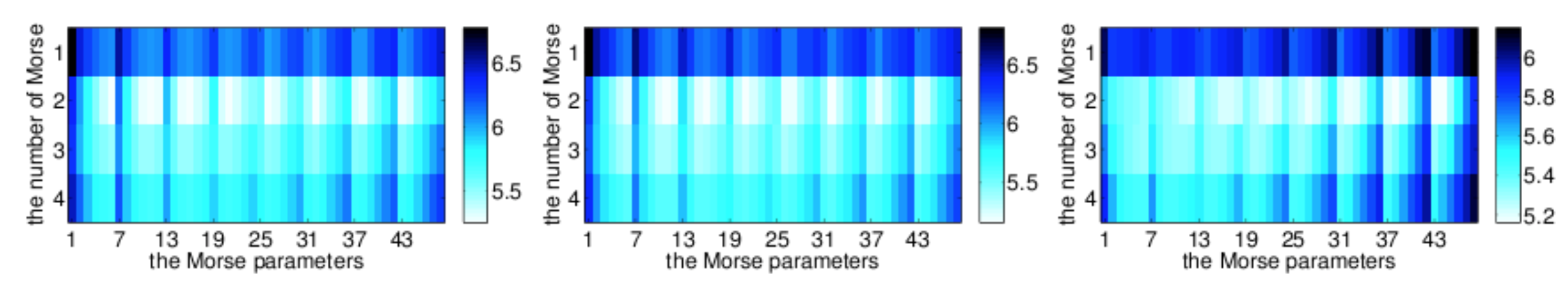}
\end{centering}
\caption{\label{fig:ESM-4.5} Exploring the parameter space: heat maps visualizing the OT distance between the itvPS for a clean signal and the ConceFT-based tvPS of noisy versions, with SNR of 0 dB, for a two-component signal $s^\#$ in $\mathcal{C}$, and for three different types of additive noise: Gaussian (left), ARMA(1,1) (middle) and Poisson (right). Each heat map shows the results for the 192 different parameter combinations described in the text. The color of each box represents the OT distance: lighter colors indicate better performance. }
\end{figure}

When we computed similar heat maps for other randomly picked signals in $\mathcal{C}$, the results were virtually identical. Our exploration showed that the combination $\beta=30,\,\gamma=9$, $J=2$ lead to the best performance; we thus chose these values for the remainder of the paper.
 
\subsection*{ESM-4e. ConceFT results for noisy signals} 
Figure \ref{fig:ESM-4.6} is the analog of Figure 8 in the main paper, for $s$, the signal used to calibrate the parameter $N$ for the ConceFT algorithm rather than the ``new'' signal $s^{\ast}$. 

\begin{figure}[h!]
\begin{centering}
\includegraphics[width=\textwidth]{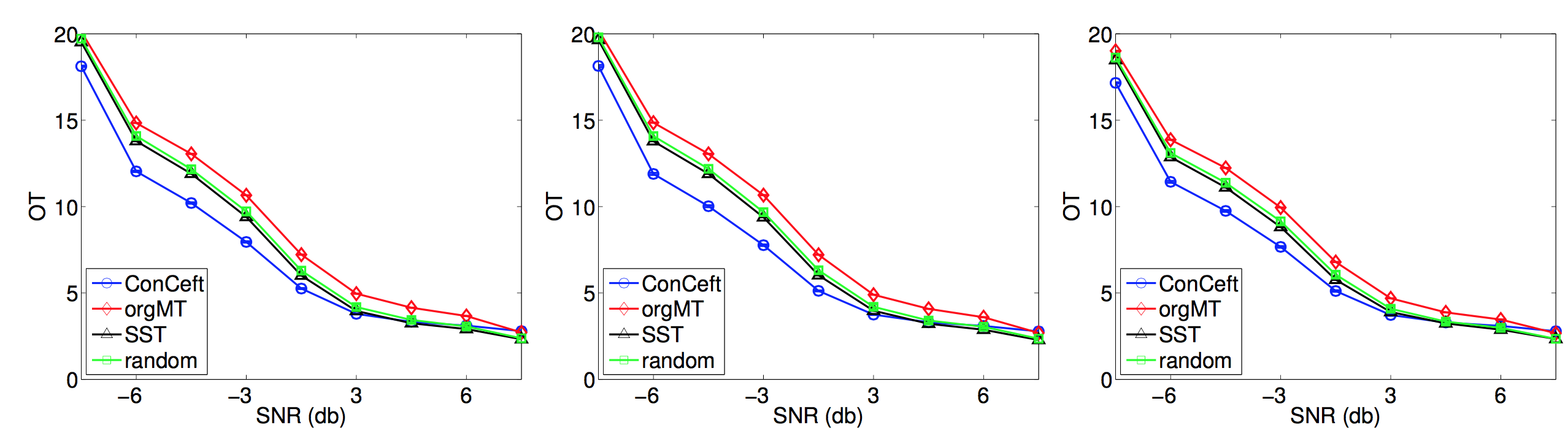}
\end{centering}
\caption{\label{fig:ESM-4.6} OT distance of CWT-based ConceFT results against signal to noise ratio (SNR) for the signal $s(t)$, and comparison with standard SST (with respect to the lowest-order Morse wavelet in black, and a random combination of the first two Morse wavelets in green) and standard multi-taper SST, both also CWT-based (see text). Noise type (left to right): Gaussian, ARMA(1,1), and Poisson. The ConceFT result is the mean OT-distance for 20 independent ConceFT computations; the standard deviation is smaller than the size of the marker.} 
\end{figure}

For each type of noise (Gaussian, ARMA(1,1) or Poisson) and each SNR considered (i.e. $x$ dB, where $x \in \{-7,-6,\ldots,6,7\}$), 20 independent realizations of the noise process are considered; for each of the resulting noisy signals the ConceFT analysis and the OT-distance of the resulting tvPS to the itvPS of the clean signal are computed; the mean and the standard deviation for each are shown in  Figure ESM-4.6. 

Figure \ref{fig:ESM-4.6} also compares the ConceFT results with those of simple SST (using 
either the first Morse wavelet with parameters $\beta=30,\,\gamma=9$ as reference wavelet,
or {\em one} random linear combination of the two first Morse wavelets) and of multi-taper SST
(denoted as \textsf{orgMT}), using the same $\psi_j$ as ConceFT. For each of these alternate methods,
we likewise computed the mean OT-distance of the tvPS to the itvPS for 20 noise realizations. 
(Note that the results are very similar to those in Figure 8 in the main paper, for $s^{\ast}$.)

\subsection*{ESM-4.f ConceFT with STFT\label{sec:STFT}}

As a complement to Figure 2 of the main paper, which shows STFT-based ConceFT results and compares them with other SST-based algorithms (simple STFT-based SST with a Gaussian window, or multi-taper SST), we show below the results of STFT-based ConceFT for the same signals $s$ and $s^{\ast}$ for which the main paper showed, in Figure 7, CWT-based ConceFT results. Before we do this, we give some more details about how these STFT-based ConceFT results are obtained. 

We explain in the main paper that SST can be defined starting from a STFT just as well as from a CWT. The whole ConceFT analysis, theoretical as well as numerical, can be carried out equally well using such STFT-SST representations. At the start of Section ESM-4, above, we explained the rationale for choosing Morse functions for the $\psi_j$; this rationale leads similarly to the choice of Hermite functions as the natural basis windows $h_j$ for STFT-based ConceFT. As in the CWT case, the choice of the family is completely fixed by the values of two parameters; in the STFT case these correspond to the eccentricity of the elliptic localization domain in the TF-plane and the tilt of the major axis of this ellipse with the time-axis (see Figure \ref{fig:ESM-4.1} bottom-left). Since there is no a priori reason to expect that chirping the Hermite functions in any direction (which tilting the elliptic domain would lead to) gives any advantage, we left this parameter out of consideration. The remaining parameter then corresponds to scaling the Hermite functions.

In analogy with the analysis in subsection EMS-4c, we thus explored the OT-distance of the STFT-based ConceFT tvPS of signals in $\mathcal{C}$ to their itvPS, for different rescalings and different numbers $J$ of Hermite functions. Minimizing this OT-distance led us to picking Hermite functions for which the underlying Gaussian function was scaled so that the bandwidth $h=5/16$; that is, the Gaussian function is $\frac{4}{\sqrt{5\sqrt{\pi}}}e^{-128t^2/25}$ (when measured in samples, since the sampling rate is $160$Hz, this corresponds to an effective width of $600$ samples, or $3.75$ sec, for the window functions $h_j$); the optimal number $J$ of functions was $4$. We then kept these parameter choices for our further experiments. The number $N$ of randomly picked linear combinations of the window functions was taken to be $20$, as in the CWT case. 

Once all the parameters are fixed, we can use the calibrated STFT-based ConceFT approach to study noisy versions of signals in $\mathcal{C}$. To compress dynamical range of the tvPS plots, we use the same trick as for Figure 7 of the main paper: we first reduce all the tvPS to the same total ``energy'', by multiplying each discretized tvPS $\boldsymbol{\widetilde{\texttt{P}}}_Y \in \RR^{m \times n}$ with an appropriate constant so that 
the ``mean energy'' of all entries equals the same number for all subfigures; that is, for some $\theta>0$ so that $\frac{1}{nm}\sum_{k=1}^m\sum_{l=1}^n \big(\boldsymbol{\widetilde{\texttt{P}}}_Y\big)_{k,l} = \theta$. We take $\theta=5$, as in the main paper, for Figures 7 and 9. Then we plot $\boldsymbol{R}\in\RR^{m\times n}$ rather than $\boldsymbol{\widetilde{\texttt{P}}}_Y\in\RR^{m\times n}$ itself, where $\boldsymbol{R}_{k,l}:=\log(1+\min\{(\boldsymbol{\widetilde{\texttt{P}}}_Y)_{k,l},q\})$, $k=1,\ldots,m$, $l=1,\ldots,n$ and $q$ is the same cut-off as used in the main paper in Figures 7 and 9 (ensuring that the gray-scale value plots are all comparable).

\begin{figure}[h!]
\begin{centering}
\includegraphics[width=\textwidth]{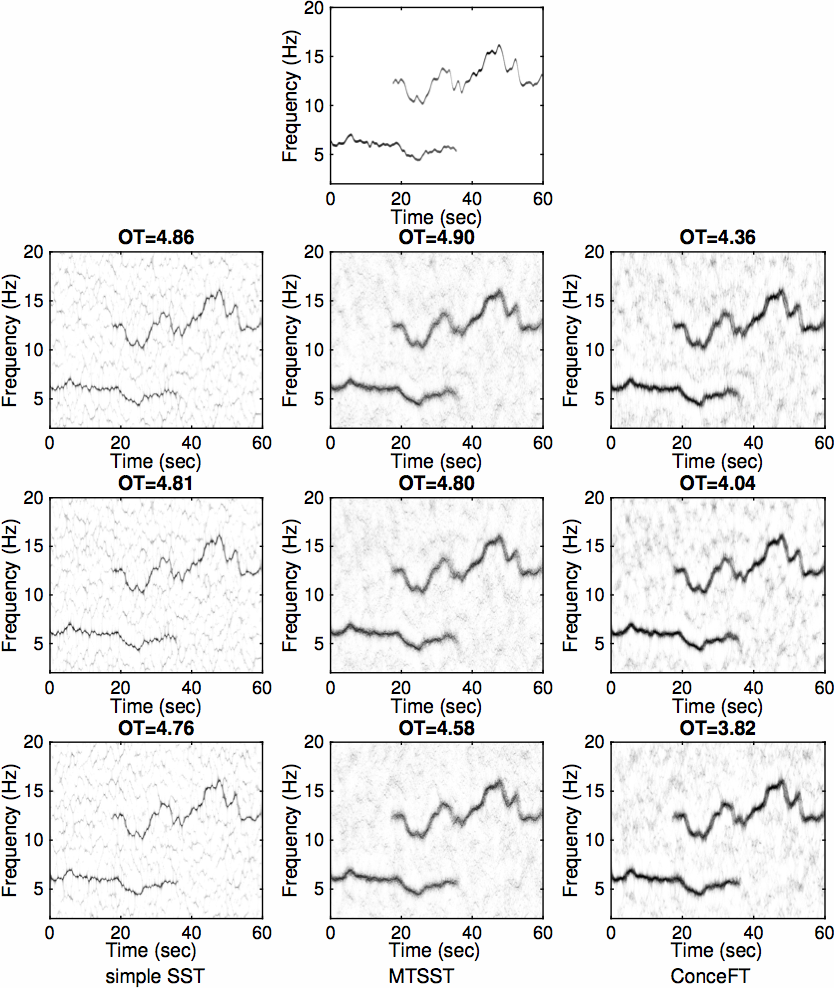}
\end{centering}
\caption{\label{fig:ESM-4.7} STFT-based results for noisy versions of the same signal $s$ as in Figures 3, 4 and the top half of Figure 7 in the main paper; the noisy versions considered are also the same realizations as in Figure 4 in the main paper. Top: itvPS of the clean signal $s$; next three columns: the tvPS of $Y$ with different noises. Each column corresponds to one algorithm; each row to one noise type. Noise types: from top to bottom, in order: Gaussian, ARMA$(1,1)$ and Poisson noise, in each case with SNR of 0 dB. Different approaches: from left to right, in order: SST using a Gaussian window with the bandwidth $h=5/16$; STFT-based multi-taper-SST (using the top 6 Hermite functions: the same Gaussian again, and the next 5 Hermite functions); STFT-based ConceFT (using 20 random combinations of the top 4 Hermite functions). 
}
\end{figure}

Figure \ref{fig:ESM-4.7} shows the results for STFT-based SST, with a Gaussian window with the bandwidth $h=5/16$, on the (discretized) signal $Y(t_k) = s(t_k) + \sigma \xi_k$, where the $\xi_k$ are i.i.d. realizations of a Gaussian noise process, and $\sigma$ is chosen so that the SNR equals 0 dB; next, it shows the tvPS obtained with multi-taper SST, using the first 6 Hermite functions as window tapers (note that the Gaussian used for the simple SST is included among these; it is the lowest-order Hermite function); finally it shows the result of ConceFT, averaging SST results using 20 random linear combinations of those same 6 Hermite functions. In all three cases the OT-distance to the itvPS is given as well. Figure \ref{fig:ESM-4.8} is entirely similar, but now for the signal $s^{\ast}$ rather than $s$.

\begin{figure}[h!]
\begin{centering}
\includegraphics[width=\textwidth]{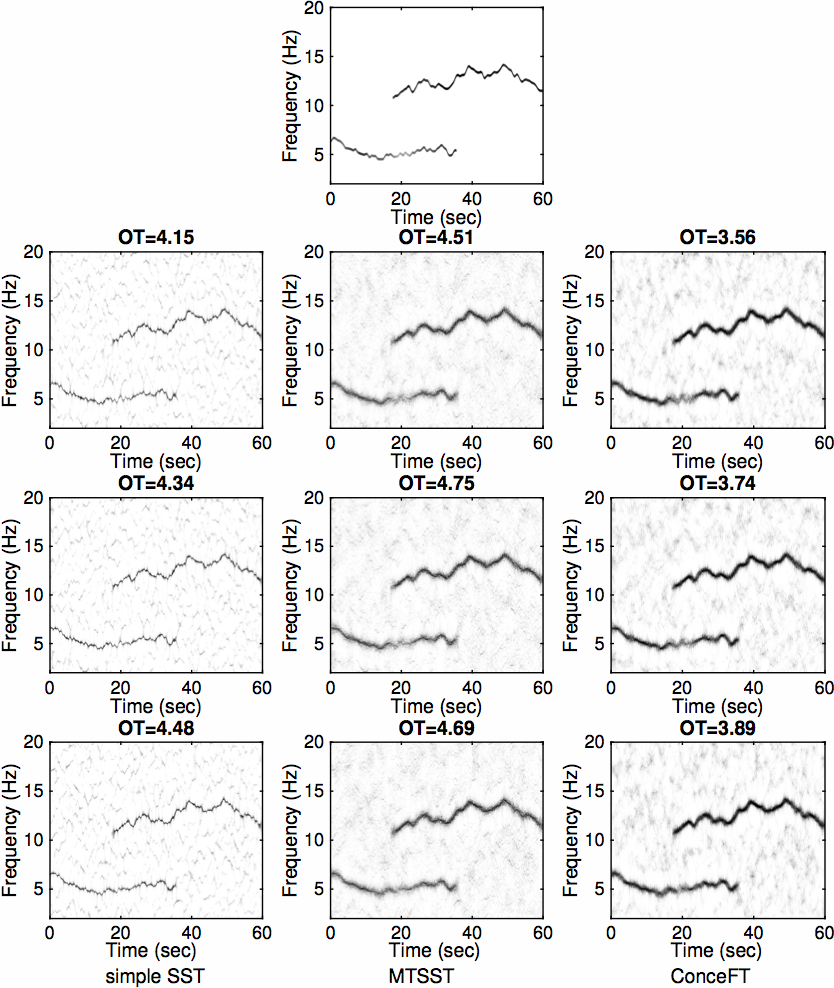}
\end{centering}
\caption{\label{fig:ESM-4.8} 
STFT-based results for noisy versions of the same signal $s^{\ast}$ as in Figures \ref{fig:ESM-4.2}, \ref{fig:ESM-4.3} and the bottom half of Figure 7 in the main paper; the noisy versions considered are also the same realizations as in Figure \ref{fig:ESM-4.3}. Top: itvPS of the clean signal $s^{\ast}$; next three columns: the tvPS of $Y^{\ast}$ with different noises. Each column corresponds to one algorithm; each row to one noise type. Noise types: from top to bottom, in order: Gaussian, ARMA$(1,1)$ and Poisson noise, in each case with SNR of 0 dB. Different approaches: from left to right, in order: SST using a Gaussian window with the bandwidth $h=5/16$; STFT-based multi-taper-SST (using the top 6 Hermite functions: the same Gaussian again, and the next 5 Hermite functions); STFT-based ConceFT (using 20 random combinations of the top 4 Hermite functions). }
\end{figure}

For Figures 7 and 9 in the main paper, and Figures \ref{fig:ESM-4.7} and \ref{fig:ESM-4.8}, we
used the {\sc Matlab} command {\tt imagesc} to generate the plots of the different tvPS 
$\widetilde{\texttt{P}}_Y$, 
in the form {\tt imagesc$(\log\left[1+\widetilde{\texttt{P}}_Y\right]$, $[\, 0 ~ \log(1+q)\, ])$}, where $q>0$; 
the two-entry array $[\,0 ~ \log(1+q)\,]$ in this expression ensures that
the gray scale value at point $(t,\xi)$ in the plot is linearly proportional to
the value of $\log\left[1+\max\left(\widetilde{\texttt{P}}_Y(t, \xi),q\right)\right]$, 
with {\tt white} standing for $0$ and {\tt black} for $q$. 
The same value of $q$ is used for all the tvPS plots.
Plotting $\log(1+\widetilde{\texttt{P}}_Y)$ rather than $\widetilde{\texttt{P}}_Y$ itself
makes it possible to display a wider dynamical range;
fixing the full gray-scale range to cover exactly $[\,0 ,~\log(1+q)\,]$ in each plot 
ensures that the figures present a fair visual comparison of the different tvPS. 
The maximum $q$ also functions as a saturation cut-off: all values of 
$\widetilde{\texttt{P}}_Y(t,\xi)$ exceeding $q$ are rendered as black in the plots, regardless
of the excess. The numerical value of $q$ was picked so that the saturation cut-off
is active on only an exceedingly low number of outliers.

Note that we have systematically normalized the  $\widetilde{\texttt{P}}_Y$ to have the same mean, which we picked to be $5$ here. 
This value is not completely arbitrary: we picked it so that the
dynamical-range compressing function $\log(1+\tilde{\texttt{P}}_Y)$ makes the noise artifacts clearly visible. 
Different values are possible, and the choice depends on the applications; different types of signals and different desired visualization characteristics, typically correspond to different choices for $\theta$.

The other issue is the determination of a ``natural'' or ``good'' value for the cut-off $q$. 
For the purposes of this paper, where we wanted to give a visualization of the goodness-of-fit
to the itvPS of a signal $s$
of the tvPS for different noisy versions $Y$ 
(obtained by adding to $s$ different types of noise, possibly also of
different strength), and compare these for different analysis methods and different 
noise types/strengths, it is most natural to fix a uniform value of $q$ (after uniform normalization
of the $\widetilde{\texttt{P}}_Y$). When ConceFT is used in practice, however, we expect 
to use one particular analysis method, to have at hand only one realization of the
(unknown) noise process, and (of course) not to have a ground truth with which to compare. 
An important role of $q$, for the rescaling used by {\tt imagesc} 
for the visual display, is to downplay an otherwise exaggerated impact from outliers.
To determine $q$, based only on $\widetilde{\texttt{P}}_Y$ itself, it would thus be natural
to choose it as some fixed percentile of the distribution of values of 
$\widetilde{\texttt{P}}_Y$. 

Figures \ref{fig:ESM-4.9} through \ref{fig:ESM-4.11} show the {\em same} $\widetilde{\texttt{P}}_Y$ as also plotted in Figure 7 in the main paper (for CWT) and Figures \ref{fig:ESM-4.7} and  \ref{fig:ESM-4.8} (for STFT), but with a plotting scheme that corresponds more to the realistic signal analysis situation, where we only have the signal at hand. More precisely, for Figures \ref{fig:ESM-4.9} through \ref{fig:ESM-4.11}, we do not normalize the tvPS to have the same total ``energy'', and we determine the cut-off $q$ for each $\tilde{\texttt{P}}_{Y}$ individually; we plot $\boldsymbol{R}\in\RR^{m\times n}$, which is defined as $\boldsymbol{R}_{k,l}:=\log(1+\min\{(\boldsymbol{\widetilde{\texttt{P}}}_Y)_{k,l},q\})$, $k=1,\ldots,m$, $l=1,\ldots,n$, where for each plot, the value $q$ is given by the $99.8\%$ percentile of only the 
$\widetilde{\texttt{P}}_Y$ of the transform/data for that individual plot itself. The 
difference between the two cases is striking, especially for the STFT figures. This suggests
a more thorough exploration would be useful of how to optimally pick $q$ depending on noise and signal
structure and on analysis method chosen; this is beyond the scope of this paper, however.

\begin{figure}[h!]
\begin{centering}
\includegraphics[width=\textwidth]{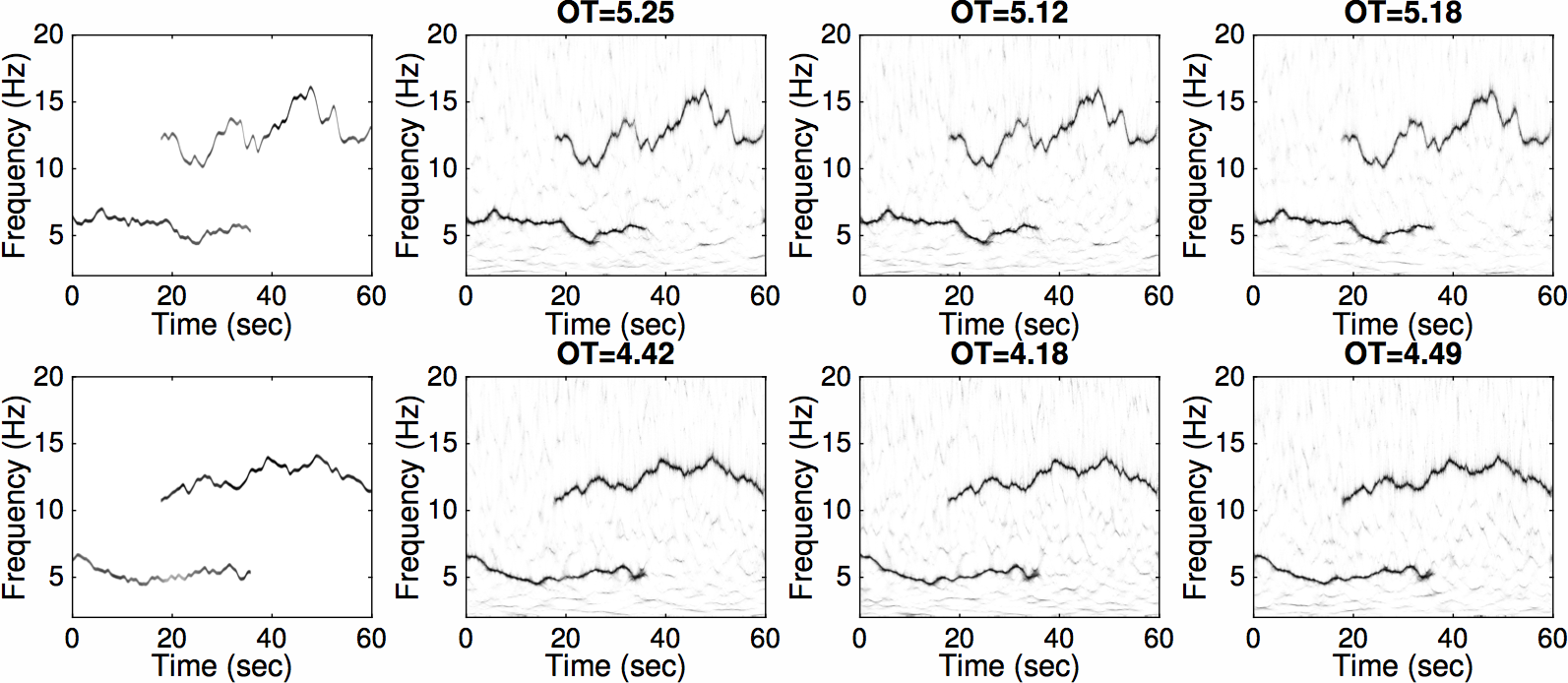}
\end{centering}
\caption{\label{fig:ESM-4.9} This figure contains CWT-based $\conceft$ of the same signal $s$ as in Figure 7 in the main paper, but with a different truncation plotting strategy (see text). 
First row: results for the signal $s$; second row: results for the signal $s^{\ast}$.  Left to right: ideal time-varying TF power spectrum (itvPS) for the clean signal, followed by results of $\conceft$ with Morse wavelets after (in order) Gaussian, ARMA(1,1) or Poisson noise was added, with SNR of 0 dB. The figures are plotted with $q$ chosen to be the $99.8\%$ quantile of each figure, and without normalizing $\tilde{\texttt{P}}_Y$. For each of the tvPS panels, the header gives the OT distance to the corresponding itvPS, which is the same as before, since the $\tilde{\texttt{P}}_Y$ are the same.}
\end{figure}

\begin{figure}[h!]
\begin{centering}
\includegraphics[width=\textwidth]{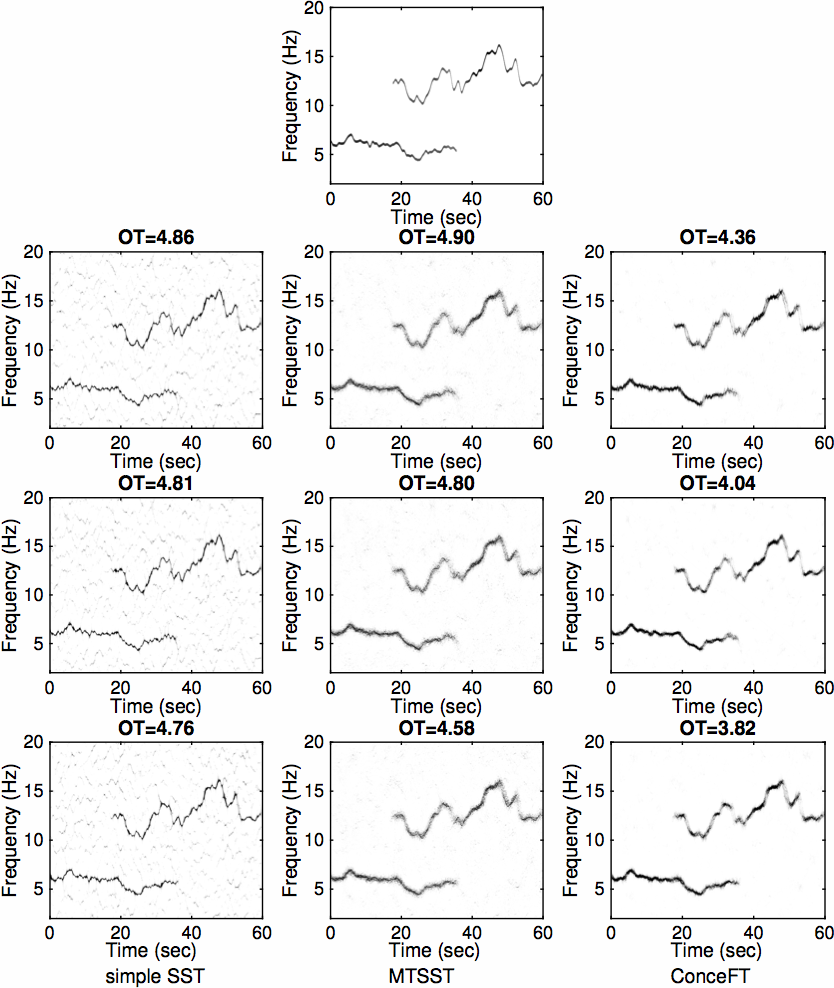}
\end{centering}
\caption{\label{fig:ESM-4.10} 
STFT-based $\conceft$ results for noisy versions of the same signal $s$ as in Figures \ref{fig:ESM-4.7} and the top half of Figure 7 in the main paper; the noisy versions considered are also the same realizations as in Figures 3 and 4 in the main paper. The plotting strategy is different, however (see text). Top: itvPS of the clean signal $s$; next three columns: the tvPS of $Y$ with different noise types. Each column corresponds to one algorithm; each row to one noise type. Noise types: from top to bottom, in order: Gaussian, ARMA$(1,1)$ and Poisson noise, in each case with SNR of 0 dB. Different approaches: from left to right, in order: SST using a Gaussian window with the bandwidth $h=5/16$; STFT-based multi-taper-SST (using the top 6 Hermite functions: the same Gaussian again, and the next 5 Hermite functions); STFT-based ConceFT (using 20 random combinations of the top 4 Hermite functions). The figures are plotted with $q$ chosen to be the $99.8\%$ quantile of each figure, and without normalizing $\tilde{\texttt{P}}_Y$. For each of the tvPS panels, the header gives the OT distance to the corresponding itvPS, which is the same as before, since the $\tilde{\texttt{P}}_Y$ are the same.}
\end{figure}

\begin{figure}[h!]
\begin{centering}
\includegraphics[width=\textwidth]{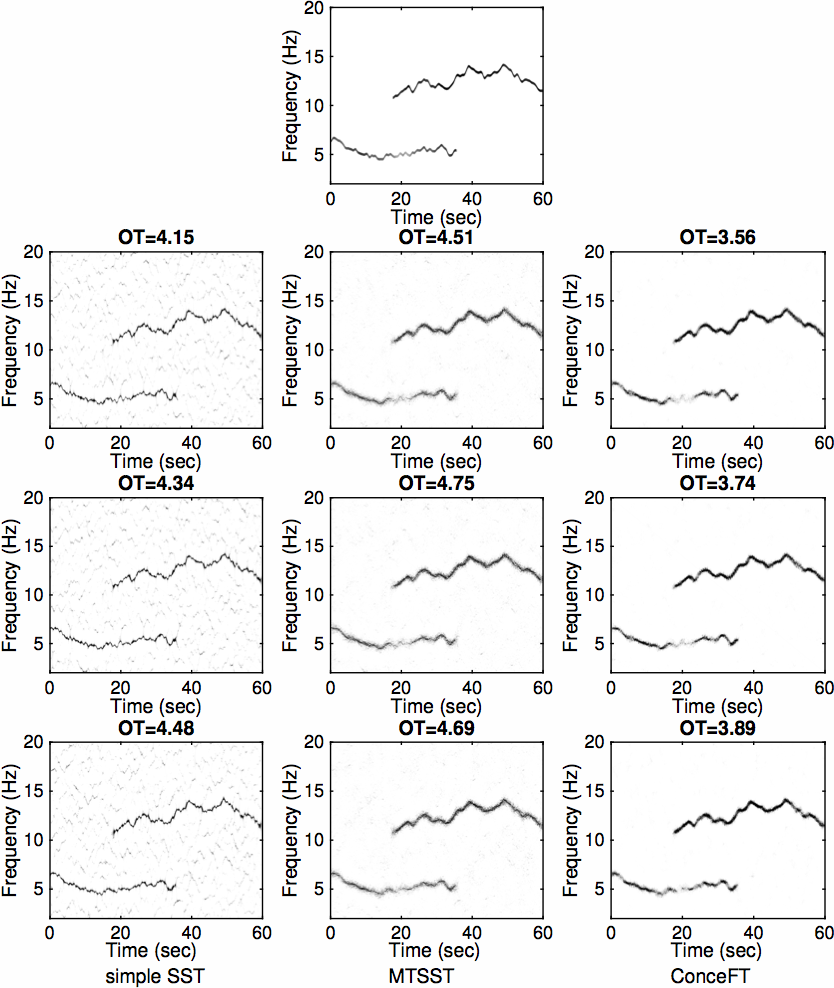}
\end{centering}
\caption{\label{fig:ESM-4.11} 
STFT-based $\conceft$ results for noisy versions of the same signal $s^{\ast}$ as in Figures \ref{fig:ESM-4.8} and the bottom half of Figure 7 in the main paper; the noisy versions considered are also the same realizations as in Figures \ref{fig:ESM-4.2} and \ref{fig:ESM-4.3}. The plotting strategy is different, however (see text). Top: itvPS of the clean signal $s^{\ast}$; next three columns: the tvPS of $Y^{\ast}$ with different noises. Each column corresponds to one algorithm; each row to one noise type. Noise types: from top to bottom, in order: Gaussian, ARMA$(1,1)$ and Poisson noise, in each case with SNR of 0 dB. Different approaches: from left to right, in order: SST using a Gaussian window with the bandwidth $h=5/16$; STFT-based multi-taper-SST (using the top 6 Hermite functions: the same Gaussian again, and the next 5 Hermite functions); STFT-based ConceFT (using 20 random combinations of the top 4 Hermite functions). The figures are plotted with $q$ chosen to be the $99.8\%$ quantile of each figure, and without normalizing $\tilde{\texttt{P}}_Y$. For each of the tvPS panels, the header gives the OT distance to the corresponding itvPS, which is the same as before, since the $\tilde{\texttt{P}}_Y$ are the same.}
\end{figure}

\end{document}